\documentclass[12pt]{article}
\usepackage{microtype}
\usepackage{graphicx}
\usepackage{subfigure}
\usepackage{amsmath}
\usepackage{bbm}
\usepackage{amssymb}
\usepackage{mathtools}
\usepackage[shortlabels]{enumitem}
\usepackage{booktabs} %
\usepackage{amsmath,amssymb,amsfonts}
\usepackage{setspace}
\usepackage{graphicx}
\renewcommand\labelenumi{(\roman{enumi})}
\renewcommand\theenumi\labelenumi
\usepackage{amsthm}

\usepackage{bbm}

\usepackage{textcomp}
\usepackage{xcolor}
\usepackage{amsmath}

\usepackage{multicol}

\newcounter{relctr} %
\everydisplay\expandafter{\the\everydisplay\setcounter{relctr}{0}} %

\AtBeginDocument{} %

\newtheorem{theorem}{Theorem}%
\newtheorem{lemma}[theorem]{Lemma}
\newtheorem{corollary}{Corollary}[theorem]
\newtheorem{proposition}[theorem]{Proposition}
\newtheorem{remark}{Remark}[section]

\newcommand{\abs}[1]{\left\lvert{#1}\right\rvert}

\newcommand{\N}{\mathbb{N}}
\newcommand{\R}{\mathbb{R}}

\newcommand{\E}{\mathbb{E}}

\usepackage{natbib}

\newcommand{\beq}{\begin{eqnarray*}}
\newcommand{\eeq}{\end{eqnarray*}}
\newcommand{\beqn}{\begin{eqnarray}}
\newcommand{\eeqn}{\end{eqnarray}}

\newcommand{\eps}{\varepsilon}

\newcommand{\mathe}{\mathrm{e}}
\newcommand{\hide}[1]{}

\usepackage{fancyvrb} %

\newcommand{\set}[1]{\left\{ #1 \right\}}

\newcommand{\PR}[2][]{\mathop{\mathbb{P}}_{#1}\left( #2 \right)}
\renewcommand{\P}{\mathbb{P}}

\newcommand{\inv}{^{-1}}

\DeclareMathOperator*{\argmax}{arg\,max}

\DeclareMathOperator{\Bin}{Bin}

\newcommand{\paren}[1]{\left( #1 \right)}
\newcommand{\sqprn}[1]{\left[ #1 \right]}
\newcommand{\nrm}[1]{\left\Vert #1 \right\Vert}

\newcommand{\vertiii}[1]{{\left\vert\kern-0.25ex\left\vert\kern-0.25ex\left\vert #1 
    \right\vert\kern-0.25ex\right\vert\kern-0.25ex\right\vert}}

\def\longto{\mathop{\longrightarrow}\limits}
\newcommand{\ninf}{\longto_{n\to\infty}}

\usepackage{hyperref}

\newcommand{\blind}{1}

\begin{document}

\def\spacingset#1{\renewcommand{\baselinestretch}%
{#1}\small\normalsize} \spacingset{1}

\if1\blind
{
  \title{\bf Distribution Estimation under the Infinity Norm}
  \author{Aryeh Kontorovich \\
    Ben Gurion University, Israel\\
    and \\
    Amichai Painsky\\
    Tel Aviv University, Israel}
  \maketitle
} \fi

\bigskip

\begin{abstract}
We present novel bounds for
estimating discrete probability distributions under the
$\ell_\infty$ norm.
These are nearly optimal in various precise senses,
including a kind of instance-optimality.
Our data-dependent convergence guarantees for the maximum likelihood estimator
significantly improve upon 
the
currently known results. 
A variety of techniques are utilized
and innovated upon, including
Chernoff-type inequalities and empirical Bernstein bounds. 
We illustrate our results in synthetic and real-world experiments. Finally, we apply our proposed framework to a basic selective inference problem, where we estimate the most frequent probabilities in a sample. 
\end{abstract}

\noindent%
\vfill

\newpage
\spacingset{1.5} %

\section{Introduction}\label{Introduction}

Consider a probability distribution $p$ over 
$\N=\set{1,2,\ldots}$.
Let $X^n$ be a sample of $n$ independent observations from $p$.  In this work we study the basic problem of estimating $p$ from $X^n$. We focus our attention to the infinity norm, which is formally defined in (\ref{infinity norm}). 
Also known as the 
{\em uniform} or {\em supremum} norm,
this popular metric over distributions
has a number of important applications
---
in addition to being a fundamental object
of independent interest
\citep{boucheron2003concentration,van2014probability}. %
Among the applications is
a selective inference scheme for multinomial proportions, as 
discussed below.

Our reference point is the following simple and classic
bound,
whose proof 
is an easy consequence of McDiarmid's inequality:
for all $\delta\in(0,1)$,
\beqn
\label{eq:baseline}
\sup_{i\in\mathbb{N}} |p_i - \hat{p}_i(X^n)| 
\lesssim
\sqrt{
\frac{\log(1/\delta)}{n}
}
\eeqn
holds with probability at least $1-\delta$,
where
$\hat{p}_i(X^n)$
is the 
maximum likelihood estimator (MLE) 
defined below and $\lesssim$ hides small
absolute constants.
This rate is known to be tight in the worst case (Lemma \ref{lem:bern-np}),
but can certainly be improved upon for benign distributions. For example, when
$p=(1-\theta,\theta)$ is the Bernoulli distribution, Bernstein's inequality
\citep[Corollary 2.11]{boucheron2003concentration}
yields
\beqn
\label{eq:bern}
|\theta-\hat\theta|
\lesssim
\sqrt{\frac{\theta(1-\theta)}
{n}\log\frac1\delta
}
+
\frac1n\log\frac1\delta,
\eeqn
where $\hat\theta$ is the MLE.
Furthermore,
\eqref{eq:bern} has an
{\em empirical Bernstein}
version \citep[Lemma 5]{DBLP:conf/icml/DasguptaH08}, 
in which the unknown
quantity $\theta$
is replaced 
in the right-hand side
by the empirically
computable $\hat\theta$.

Drawing inspiration from
\eqref{eq:bern} and its empirical version,
we might expect something like
\beqn
\label{eq:dream}
\sup_{i\in\mathbb{N}} 
|p_i - \hat{p}_i(X^n)| 
\stackrel{?}{\lesssim}
\sqrt{\frac{
v^*
\log\frac1\delta}{n}
}
+
\frac1n\log\frac1\delta,
\eeqn
where $v^*=
\max_{i\in\mathbb{N}}p_i(1-p_i)
$,
or, even more ambitiously,
some version of
\eqref{eq:dream}
with $v^*$ replaced by its
empirical version $\hat{v}^*$.
It will turn out that \eqref{eq:dream}
is
too optimistic.
Absent an oracle that tells us the index of the largest mass, %
some additional cost must be incurred for estimating
many symbol probabilities simultaneously.

\paragraph{Our contributions.}
Our main result amounts
to nearly achieving the 
ultimate
goal.
{First, we derive a Chernoff-type upper bound in Thereom \ref{T_data_independent}, which improves upon \eqref{eq:baseline}. Theorem \ref{main_result_1} introduces its data-dependent counterpart, which demonstrates a significant improvement in small sample regimes.}
Next, we establish in Theorem~\ref{thm:aryeh-ub}
a version of \eqref{eq:dream}
where $v^*$ is replaced 
by $v^*\log\frac1{v^*}$
and provide even sharper bounds therein. 
These are matched by nearly optimal lower bounds,
in distinct senses made precise below.
Finally, 
we apply our results to the
important 
problem of
selective inference. Specifically, we study the basic problem of inferring the most frequent events in a sample and achieve a significant improvement over currently known schemes.

\section{Definitions and Problem Statement}\label{Definitions}
Consider a probability distribution $p$ over $\mathbb{N}$, which 
induces the random variable $X\sim p$. 
The {\em support size}
of $X$, $||p||_0$, 
is also the {\em alphabet size}
--- and unless stated otherwise, our results
hold even when these are infinite.
Let
$X^n = (X_1,...,X_n)$
be a sample consisting of $n$
independent copies of $X$.
Let $c_i(X^n)$ be the count (number of appearances) of the $i^{th}$ symbol in the sample.
Let $\hat{p}(X^n)$ be the maximum likelihood estimator (MLE) of $p$;
namely,
$\hat{p}_i(X^n)=c_i(X^n)/n$  for every $i\in \mathbb{N}$. In this work, we study 
empirical 
distribution
estimation of $p$ under the infinity norm. That is, given a prescribed $\delta>0$, we seek a random variable $T_\delta(X^n)$ such that  
\begin{align}\label{infinity norm}
||p-\hat{p}(X^n)||_\infty\triangleq 
\sup_{i\in\mathbb{N}} |p_i - \hat{p}_i(X^n)| 
\leq T_\delta(X^n)
\end{align}
with probability of at least $1-\delta$.
In light of \eqref{eq:baseline},
we also require that, for fixed $\delta$,
$T_\delta(X^n)\to0$
as $n\to\infty$, in some appropriate sense.

Notice that (\ref{infinity norm}) may also be viewed as a multinomial inference scheme. Specifically, (\ref{infinity norm}) implies that $p_i \in [\hat{p}_i \pm T(X^n)]$ simultaneously for all $i \in \mathbb{N}$. This \textit{confidence region} (CR) defines a hypercube around the MLE, which covers $p$ with a \textit{confidence level} of $1-\delta$.

\section{Related Work}\label{Related Work}
Discrete probability estimation is a fundamental problem in many fields. It is extensively studied under a variety of merits such as total variation  \citep{jiao2017maximum,CohenKW20}, KL divergence \citep{orlitsky2015competitive}, Hellinger distance \citep{hellinger1909neue} Wasserstein metric \citep{kantorovich1960mathematical}, Kolmogorov-Smirnov distance \citep{smirnov1948table} and others. The interested reader is referred to \citep{rice2006mathematical,painsky2019bregman,painsky2023quality} for a comprehensive discussion. In this work we focus on the infinity norm. 
Here, the baseline is the bound 
implicit in
\eqref{eq:baseline},
where, in the language of
\eqref{infinity norm},
$T_\delta(X^n)=
\sqrt{{1}/{n}}+\sqrt{{\log(1/\delta)}/{2n}}
$.

The infinity norm is difficult to analyze in the general case (\ref{infinity norm}). In fact, it is later shown (Section \ref{Main Results}) that (\ref{eq:baseline}) is only asymptotically 
tight and only for the worst-case distribution,
and can be significantly improved in a limited sample regime. On the other hand, the binomial case, $||p||_0=2$, is fairly understood 
as far as minimax optimal and fully empirical
bounds.
In particular, if $Y \sim \text{Bin}(n,\theta)$ is a Binomial random variable 
and
$\hat{\theta}=Y/n$ its MLE,
then
\citet{bousquet2003introduction} and later \citet{dasgupta2008hierarchical} showed that
\begin{align}\label{binomial1}
    |\theta-\hat{\theta}|\leq \sqrt{\frac{5\hat{\theta}(1-\hat{\theta})}{n}\log\frac{2}{\delta}}+\frac{5}{n}\log \frac{2}{\delta}
\end{align}
with probability of at least $1-\delta$. A closely related line of work appears in the statistics literature. Let $F(y;n,\theta)$ be the cumulative distribution function of $Y$. For a given $y$ and $n$, let $\theta_{l}$ and $\theta_{u}$ be the solutions (with respect to $\theta$) of $F(y;n,\theta)=\delta/2$ and $F(y;n,\theta)=1-\delta/2$ respectively. \citet{clopper1934use} showed that 
\begin{align}\label{CP}
    \P\left(\theta\in[\theta_l,\theta_u]\right)\geq 1-\delta
\end{align}
for every $\theta\in[0,1]$. The interval $[\theta_l,\theta_u]$ is widely known as the \textit{exact} Clopper-Pearson (CP) confidence interval (CI). The exact notion refers to the fact that (\ref{CP}) holds for every $n$, as opposed to alternative approximations. In fact, CP is also known to be shortest possible CI, for most setups of interest. Specifically, let $\mathcal{T}$ be a collection of intervals $[t_l,t_u]$ that satisfy $\P\left(\theta\in[t_l,t_u]\right)\geq 1-\delta$, for every  $[t_l,t_u] \in \mathcal{T}$.
The shortest CI for $\theta$ is defined as the intersection of all intervals in $\mathcal{T}$. This notion also implies minimal expected length and minimal false coverage probability, uniformly. \citet{wang2006smallest} showed 
that for $n\geq \log(\delta/2)/\log 0.5$, the CP CI is the shortest. Notice that for a nominal level of $\delta=0.05$, this condition corresponds to $n\geq 6$. Hence, for practical setups of interest, the CP interval $[\theta_{l},\theta_{u}]$ is the shortest possible CI for $\theta$. Unfortunately, CP does not hold a closed-form expression. Yet, \citet{thulin2014cost} showed that for every $y\in\{1,...,n-1\}$, 
\begin{align}\nonumber
&\theta_l=\hat{\theta}-z_{\frac{\delta}{2}}\sqrt{\frac{\hat{\theta}(1-\hat{\theta})}{n}}+\frac{1}{3n}\left(\left(1-2\hat{\theta}\right)z_{\delta/2}^2-1-\hat{\theta}\right)\\\nonumber
&\theta_u=\hat{\theta}+z_{\frac{\delta}{2}}\sqrt{\frac{\hat{\theta}(1-\hat{\theta})}{n}}+\frac{1}{3n}\left(\left(1-2\hat{\theta}\right)z_{\delta/2}^2+2-\hat{\theta}\right)
\end{align}
up to additive terms of order $n^{-3/2}$, where  $z_\delta/2$ is the upper $\delta/2$ quantile of the standard normal distribution. This result implies that for every $y \in \{1,...,n-1\}$, the shortest possible CI length for $\theta$ is  
\begin{align}\label{exact_binomial}
\theta_u-\theta_l =2z_{\delta/2}\sqrt{{\hat{\theta}(1-\hat{\theta})}/{n}}+{1}/{n}+O(n^{-3/2}).
\end{align}
Moreover, we have 
\begin{align}\label{binomial_exact_2}
   |\theta-\hat{\theta}|\leq \max\{|\hat{\theta}-\theta_l|,|\hat{\theta}-\theta_u|\} 
\end{align}
with probability of at least $1-\delta$. Importantly, it can be shown that $z_{\delta/2}$ behaves asymptotically like $\sqrt{2\log(2/\delta)}$.   Comparing (\ref{eq:baseline}) to (\ref{binomial_exact_2}) (and (\ref{binomial1})) , we observe that its sample complexity, $1/\sqrt{n}$, is tight. However, there may still be room for improvement by utilizing a data-dependent scheme.

The Clopper-Pearson interval (\ref{binomial_exact_2}) provides a tight solution for the binomial case $||p||_0=2$. Yet, the problem becomes more involved in the multinomial setting (\ref{infinity norm}). Currently known methods focus on two basic regimes. The first considers an asymptotic setup, where $n$ is much greater than the alphabet size \citep{quesenberry1964large,goodman1964simultaneous,sison1995simultaneous}. The second addresses the case where both $n$ and  $||p||_0$  are small \citep{chafai2009confidence,malloy2020optimal}. Notice that while some of these methods provide rectangular CR \citep{quesenberry1964large,goodman1964simultaneous,painsky2023large,doi:10.1080/10485252.2024.2313706}, others focus on hyper-cubes \citep{sison1995simultaneous}.
Yet, all of these methods assume a finite alphabet where performance guarantees are limited to relatively small $||p||_0$. To the best of our knowledge, no method considers the case where $||p||_0$ may be infinite.

\section{Main Results}\label{Main Results}

We begin our analysis by considering a data-indepedent bound under the infinity norm. Our proposed bound generalizes (\ref{eq:baseline}) by utilizing a Chernoff-like concentration bound.
\begin{theorem}\label{T_data_independent}
     Let $p=p_{i\in \mathbb{N}}$ be a distribution over $\mathbb{N}$. Let $X^n$ be a sample of $n$ independent observations from $p$. Let $\hat{p}(X^n)$ be the MLE of $p$. Then, with probability $1-\delta$,
     \begin{align}\label{Data Independent T1}
            \nrm{p-\hat p}_\infty=&\sup_{i  \in \mathbb{N}} |p_i-\hat{p}_i(X^n)|\leq \frac{1}{n} \left(\frac{1}{\delta^{1/m}}\right) \sum_{k=1}^{m/2}\left( k^{m-k}n^k\sum_{i\in\mathcal{X}} p_i^k(1-p_i)^k\right)^{1/m}\leq\\\nonumber
            &\frac{1}{\sqrt{n}}\frac{\sqrt{m/2}}{\delta^{1/m}}\exp\left(-\frac{1}{2}+\frac{1}{m}\right)+O\left(\frac{1}{n^{\frac{1}{2}+\frac{1}{m}}}\right)
        \end{align}
        for every even $m>0$.
\end{theorem}
Theorem \ref{T_data_independent} relies Markov's inequality for higher-order moments of the infinity norm. In addition, it applies higher-order properties of the MLE and the binomial distribution. The detailed proof appears in Section \ref{proof1}. Next, similarly to Chernoff inequality, we minimize (\ref{Data Independent T1}) with respect to $m$ to obtain tighter convergence guarantees. Specifically, we minimize the leading term of (\ref{Data Independent T1}) to obtain
    \begin{align}\label{13}
        \min_{m \in \mathbb{R}^+} \frac{\sqrt{m/2}}{\delta^{1/m}}\exp\left(-\frac{1}{2}+\frac{1}{m}\right) = \sqrt{1+\log\left(\frac{1}{\delta}\right)}
    \end{align}
    for the choice $m^*=2\log(1/\delta)+2$ (See Appendix C). Hence the infimum of the bound is given by
    \begin{align}\label{Data Independent T2}
           \sqrt{\frac{1+\log\left({1}/{\delta}\right)}{n}}+O\left(\frac{1}{n^{\frac{1}{2}+\frac{1}{m^*}}}\right).
    \end{align}
    Unfortunately, this is not a great improvement over the benchmark \eqref{eq:baseline}. However, it is  shown in Section \ref{proof1}  that as  $m$ increases, the second inequality in 
    (\ref{Data Independent T1})
    becomes tight for a worst-case case distribution $p=[1/2,1/2,0,\dots,0]$. This distribution is quite ``unlikely'' in a large alphabet regime. On the other hand, if we assume a ``more likely'' uniform distribution over a finite alphabet size $A:=||p||_0<\infty $, we obtain 
    \begin{align}%
            \nrm{p-\hat p}_\infty\leq \frac{1}{\sqrt{n}}\frac{\sqrt{m/2}}{\delta^{1/m}}A^{-\frac{1}{2}+\frac{1}{m}}+O\left(\frac{1}{n^{\frac{1}{2}+\frac{1}{m}}}\right)\nonumber
        \end{align}
    for every even $m>0$. Minimizing the leading term with respect to $m$ yields
    \begin{align}
        \min_{m \in \mathbb{R}^+} \frac{\sqrt{m/2}}{\delta^{1/m}}&A^{-\frac{1}{2}+\frac{1}{m}} = \sqrt{\log\left(\frac{A}{\delta}\right)}\exp\left( \frac{1}{2}-\frac{1}{2}\log(A)\right),
    \end{align}
    for the choice $m^*=2\log(A/\delta)$. Notice the above vanishes with $A$. This result motivates our quest for a data-dependent bound, which considers an empirical estimate of $p$ and does not assume a worst-case distribution as in (\ref{eq:baseline}) and (\ref{Data Independent T2}). Theorem \ref{main_result_1} below improves upon (\ref{Data Independent T1}) and introduces a data-dependent bound which further accounts for $\hat{p}$.  
\begin{theorem}\label{main_result_1}
    Let $\delta_1>0$ and $\delta_2>0$. Let $m$ be a positive even number. Then, with probability $1-\delta_1-\delta_2$,
    \begin{align}\label{11}
&\nrm{p-\hat p}_\infty\leq\frac{1}{n}
\left(\frac{1}{\delta_1}\frac{n}{n-1}\left( \sum_i\sum_{k=1}^{m/2}k^{m-k}n^k(\hat{p}_i(1-\hat{p}_i))^k+\epsilon\right)\right)^{1/m}
\end{align}
for every even $m$, where
\begin{align}\nonumber%
\epsilon=\sqrt{\frac{n}{2}\log(1/\delta_2)}\cdot\sum_{k=1}^{m/2}k^{m-k}n^k\left(\frac{k}{n 4^{k-1}}+\frac{3k^3}{n^3 2^{2k-5}} \right).
\end{align}
\end{theorem}

To prove of Theorem \ref{main_result_1} we utilize the first inequality of (\ref{Data Independent T1}) with $\delta=\delta_1$. Then, we apply McDiarmind's inequality to obtain a concentration bound for $\sum_{i\in\mathbb{N}} p_i^k(1-p_i)^k$ around its empirical counterpart, with probability $1-\delta_2$. Finally, we apply the union bound to obtain (\ref{11}). The detailed proof is provided in Section \ref{proof2}. To further clarify the proposed bound we introduce the following simplified corollary, whose proof is located in Section \ref{proof2.1}.  
\begin{corollary}\label{T2_data_dependent}
     Let $\delta_1>0$ and $\delta_2>0$. Let $m$ be a positive even number. Then, with probability $1-\delta_1-\delta_2$,
\begin{align}\nonumber
&\nrm{p-\hat p}_\infty\leq\\\nonumber
&\frac{m}{\delta_1^{1/m}}
\frac{1}{n}\left( \sum_{k=1}^{m/2}\sum_i(n\hat{p}_i(1-\hat{p}_i))^k \right)^{1/m}+a\frac{m}{\delta_1^{1/m}} \left(\log\left(\frac{1}{\delta_2}\right)\right)^{1/2m} \left(\frac{1}{n^{\frac{1}{2}\left(1+\frac{1}{m}\right)}} +\frac{24}{n^{\frac{1}{2}\left(1+\frac{5}{m}\right)}}\right)
\end{align}
for every even $m$, where  $a=\sqrt{2\exp(1/\mathe)}$. Furthermore,  
$$\inf_m \frac{m}{\delta_1^{1/m}}=\mathe \log(1/\delta_1),$$
where and the infimum is obtained for a choice of $m^*=\log(1/\delta_1)$.
\end{corollary}
Let us compare Corollary \ref{T2_data_dependent} with the benchmark scheme (\ref{eq:baseline}) and our previous data independent bound (Theorem \ref{T_data_independent}). First, we notice a similar sample complexity of order of $\sqrt{1/n}$ in all the three schemes. This is not quite surprising, given (\ref{binomial1}). However,  Corollary \ref{T2_data_dependent} demonstrates an improved dependency on the underlying distribution, which now depends on $\hat{p}$ and does not assume a worst-case distribution. Unfortunately, the dependency in $\hat{p}$ is somewhat involved, and does not hold the desired form of (\ref{eq:dream}). Finally, we compare the dependency in the confidence level $\delta$. Here, the data independent bounds introduce a squared root  logarithmic dependency in $1/\delta$. On the other hand, Corollary \ref{T2_data_dependent} only attains a logarithmic dependency in $1/\delta_1$ (where $\delta_1<\delta$) for the right choice of $m$. This difference is typically negligible compared to the other terms, especially in a fixed $\delta$ regime as later discussed. 

Next, we present our second main result, which introduces a dependency in $p$ that is closer to the desired form (\ref{eq:dream}). First, we introduce some additional notation.
\paragraph{Notation.}
\newcommand{\V}{V^*}
\renewcommand{\phi}{\varphi}
For any distribution $p_{i\in\N}$,
define $v=v(p)$
by
$v_i=p_i(1-p_i)$
and $v^*=\max_{i\in\N}v_i$ as above. 
Define the functional
\beqn
\label{eq:Vdef}
\V(p) = 
\sup_{i\in\N}
v_i(p^\downarrow) \log(i+1),
\eeqn
where $p^\downarrow$
is $p$ sorted in non-increasing order.
Define
\beqn
\label{eq:phidef}
\phi(t) = t\log\frac1t,
\qquad 0\le t\le 1.
\eeqn

\begin{theorem}
\label{thm:aryeh-ub}
Let $p=p_{i\in\N}$
be a distribution over $\N$
and put
$v^*=v^*(p)$,
$V^*=\V(p)$.
For $n\ge 81$
and $\delta\in(0,1)$, we have that
\begin{align}
\nrm{p-\hat p}_\infty\leq&\label{eq:UB1}2\sqrt{\frac{V^*}{n}+\frac{v^*}{n}\log\frac2\delta}+\frac{4}{3n}\log\frac{2(n+1)}\delta+\frac{\log n}n \leq\\
&\label{eq:UB2}2\sqrt{\frac{\phi(v^*)}{n}+\frac{v^*}{n}\log\frac2\delta}+\frac{4}{3n}\log\frac{2(n+1)}\delta+\frac{\log n}n;
\end{align}
\begin{align}
&\label{eq:UB3} \nrm{p-\hat p}_\infty\leq 2\sqrt{\frac{v^*\log (n+1)}{n}+\frac{v^*}{n}\log\frac2\delta}+\frac{4}{3n}\log\frac{2(n+1)}\delta+\frac{\log n}n
\end{align}
holds with probability
at least $1-\delta-81/n$.
\end{theorem}
\begin{remark}
\label{rem:logn}
It seems that the $\log(n)/n$ term
can be improved to
$\log(n)/(n\log\log\log n)$;
we shall explore this in the sequel.
\end{remark}

It is instructive to compare 
Theorem \ref{thm:aryeh-ub} with our 
ambitious desideratum
\eqref{eq:dream}.
The loosest bound therein,
\eqref{eq:UB3}, 
features the desired dependency
on $v^*$,
but at the cost of a $\log n$
factor.
The sharper bound
\eqref{eq:UB2} replaces the
$\log n$ with 
$v^*\log\frac1{v^*}$.
Finally,
\eqref{eq:UB1}
gives the optimal
(at least for the MLE, cf. Proposition
\ref{prop:decoup-lb})
quantity $\V$. 
The proof of Theorem \ref{thm:aryeh-ub} is provided in Section \ref{proof3}. It relies on techniques from empirical
process theory and large deviations.  

Next, we provide the empirical counterpart of Theorem \ref{thm:aryeh-ub}, which depends on $\hat{v}^*=\sup_{i\in\N}\hat p_i(1-\hat p_i)$.

\begin{theorem}
    \label{thm:aryeh-ub-emp}
Let $p=p_{i\in\N}$
be a distribution over $\N$. Let $\hat{p}$ be the MLE of $p$. Define
\beq
a &=& \frac{4}{3n}\log\frac{2(n+1)}\delta+\frac{\log n}n,
\\
b &=&
2\sqrt{\frac{\log (n+1)}{n}+\frac{1}{n}\log\frac2\delta}.
\eeq
Then, with probability $1-\delta-81/n$,  
\begin{align}
\label{eq:logn}
\nrm{p-\hat p}_\infty
\le
a+3b^2/2
+b\sqrt{a}
+3b\sqrt{\hat v^*}/2.
\end{align}
\end{theorem}
\begin{remark}
Note that the estimate in
\eqref{eq:logn}
is of order
$
\sqrt{\frac{\hat v^*\log n}{n}+\frac{\hat v^*}{n}\log\frac1\delta}+\frac{1}{n}\log\frac{n}\delta+\frac{\log n}n
$
--- matching, up to constants, the form of
\eqref{eq:UB3}.
\end{remark}

The proof of Theorem \ref{thm:aryeh-ub-emp} follows the proof of \citet[Lemma 5]{DBLP:conf/icml/DasguptaH08} and is left for Section \ref{proof4}. 

\paragraph{Open problem.}
The estimate in Theorem \ref{thm:aryeh-ub-emp}
gives an empirical analog
of \eqref{eq:UB3}. We conjecture that some empirical analog of \eqref{eq:UB1} should be possible as well:
a bound of the general form
$
\sqrt{\frac{\hat V^*}{n}+\frac{\hat v^*}{n}\log\frac1\delta}+\frac{1}{n}\log\frac{n}\delta+\frac{\log n}n
$.

\paragraph{Near-optimality.}
To argue the near-optimality\footnote{
We use the term ``instance-optimality'' in the spirit of
Theorems 2.3 and 2.4 of \citet{CohenKW20}: fully empirical
data-dependent bounds that cannot be significantly improved upon.
}
of the above 
bounds 
we introduce our lower bounds 
on 
$\sup_{i \in \N}|\hat p_i-p_i|$
for some fixed constant $\delta>0$;
this is equivalent to 
lower bounding
$\E\sup_{i \in \N}|\hat p_i-p_i|$. 
Understanding the correct dependence on $\delta$
is left for future work.%

For a fixed $\delta$, the upper bound in
\eqref{eq:UB1}
consists of two terms:
one of order $\sqrt{V^*/n}$
and another one of order $\log(n)/n$.
We shall argue below that 
the first is tight and the second nearly so,
albeit in different senses.

The near-optimality of the $\log(n)/n$
term is proved
in the following result, whose proof is provided in Section \ref{proof5}. 
Note
that the lower bound obtained for this term
is of a minimax type, meaning that it holds for
{\em any} estimator, not just the MLE.
\begin{proposition}
\label{prop:fano}
There is an absolute constant $c>0$ such that
the following holds for all sufficiently large $n$.
For
any
estimator $\tilde{p}(X^n)$,
there is a 
distribution $p_{i\in \mathbb{N}}$ on $\N$
such that
\beqn
\label{eq:fano-lb}
\E\nrm{p-\tilde p}_\infty
&\ge&
\frac{c\log n}{n\log\log n}
\eeqn
for all sufficiently large $n$.
\end{proposition}

Our lower bound matching
the $\sqrt{V^*/n}$ term will be limited
to the MLE, but will have the advantage
of holding pointwise for {\em any} given
distribution
---
in constradistinction to
the minimax bound in
Proposition~\ref{prop:fano},
which only holds for {\em some}
adversarial distribution. The proof of Proposition \ref{prop:decoup-lb} is provided in Section \ref{proof6}. 
\begin{proposition}
\label{prop:decoup-lb}
For any distribution  $p_{i\in \mathbb{N}}$ 
and its corresponding MLE $\hat p$,
we have
\beq
\liminf_{n\to\infty}
\sqrt n
\E
\nrm{p-\hat p}_\infty
\ge c\sqrt{V^*(p)},
\eeq
where $c>0$ is an absolute constant.
\end{proposition}
\begin{remark}
    We show in Section \ref{proof6} that the lower bound is
necessarily only asymptotic (rather than {\em finite-sample}, in the sense of holding for all nn), as a consequence of previous results  \citep{Berend2013}
\end{remark}
Finally,
the following is a straightforward
consequence of Neyman-Pearson (Lemma \ref{lem:bern-np}):
\begin{proposition}
\label{prop:p_0-lb}
For
any
estimator $\tilde{p}$
there exists a 
distribution $p_{i\in \mathbb{N}}$
such that
\beq
\E
\nrm{p-\tilde p}_\infty
\ge 
c\sqrt{\frac{v^*(p)}{n}},
\eeq
for all sufficiently large $n$,
where $c>0$ is an absolute constant.
\end{proposition}
The proof of Proposition \ref{prop:p_0-lb} is provided in Section \ref{proof7}. 
It is instructive to compare
Propositions \ref{prop:decoup-lb} and \ref{prop:p_0-lb}.
The former holds for any fixed distribution $p$
and the bound is stronger (since $V^*\ge v^*$),
but only for the MLE estimate.
The latter holds for all estimators,
but the distribution can be aversarially
chosen for each sample size $n$ and the bound is weaker.
We conjecture that the lower bound in 
Proposition \ref{prop:decoup-lb}
holds for all estimators and not just the MLE.

\section{Experiments}
Let us now demonstrate our proposed bounds. 
We focus on two benchmark distributions which represent two extreme cases. That is, we study the Zipf's law and the uniform distributions. The Zipf's law distribution is a typical benchmark in large alphabet probability estimation; it is a commonly used heavy-tailed distribution, mostly for modeling natural (real-world) quantities in physical and social sciences, linguistics, economics and others fields \citep{saichev2009theory}. The Zipf's law distribution follows $p_i={i^{-s}}/{\sum_{r=1}^A r^{-s}}$ where $A$ is the alphabet size and $s$ is a skewness parameter. We set $s=1.1$ throughout our experiments. In each experiment we draw $n$ samples from a distribution over an alphabet size $A$ to evaluate 
the proposed bounds for a given confidence level $1-\delta$. We repeat this process $10^4$ times and report the average bound and coverage rate (that is, the number of times that the infinity norm is not greater than the bound).

In the first experiment we focus on $A=100$ and  $\delta=0.05$. We examine three bounds. First, we consider the bound from Theorem \ref{main_result_1} with $\delta_1=0.99\delta$, $\delta_2=0.01\delta$. We set $m$ to minimize (\ref{11}) over the worst-case distribution (see (\ref{13})). This results in $m*=\log(1/\delta_1)+2 \approx 8$. Further, we examine Theorem \ref{thm:aryeh-ub-emp} and the benchmark bound (\ref{eq:baseline}). To further assess the tightness of our results we introduce an Oracle lower bound (OLB). The OLB knows the true distribution and evaluates the $1-\delta$ quantile of the desired infinity norm. Figure \ref{fig1} summarizes the results we achieve. First, we observe that Theorem \ref{main_result_1} outperforms both Theorem \ref{thm:aryeh-ub-emp} and the benchmark. It is also relatively close to the OLB, especially as $n$ increases. We emphasize that although Theorem \ref{thm:aryeh-ub-emp} demonstrates a steep descent, it does not outperform Theorem  \ref{main_result_1}, even for a relatively large $n=10^5$. The reason for this phenomenon is the fixed $\delta$ regime, in which Theorem  \ref{main_result_1} is favorable. Importantly, all the examined bounds attain the prescribed coverage rate as desired.

\begin{figure}[ht]
\centering
\includegraphics[width =0.7\textwidth,bb=0 200 600 585,clip]{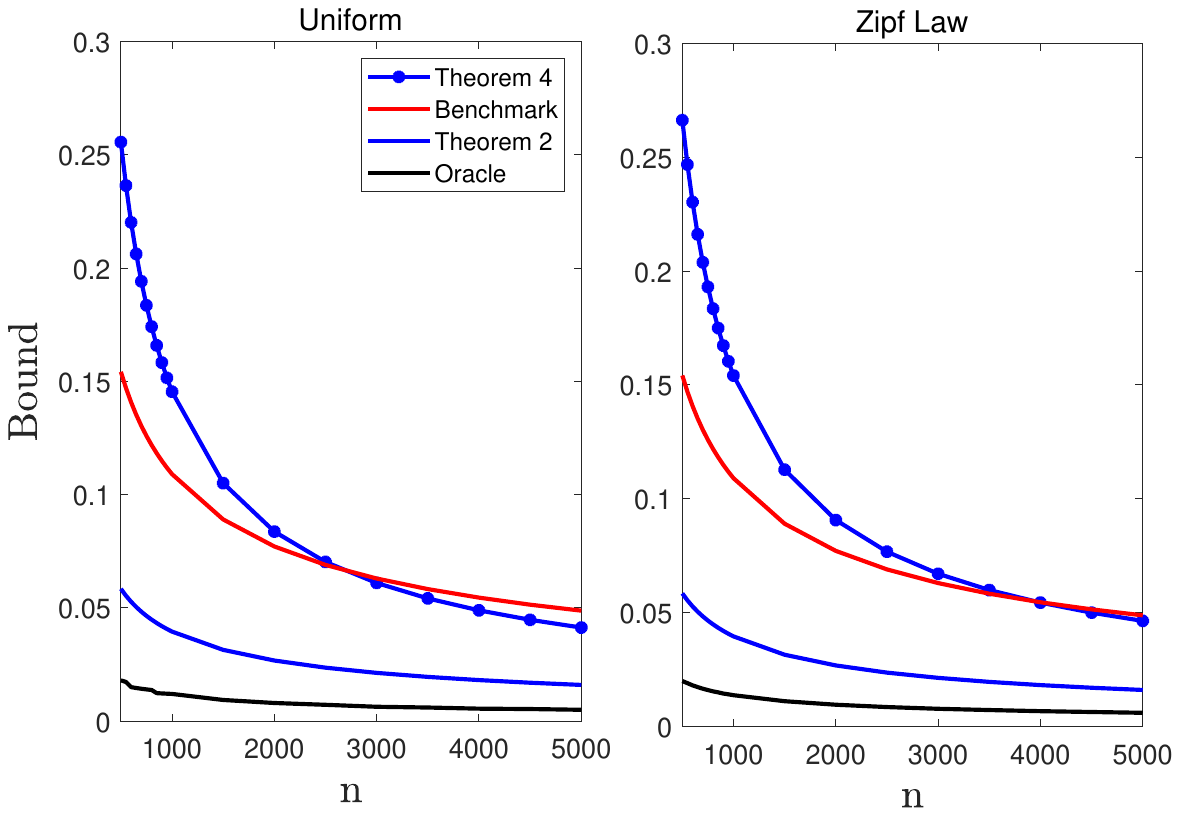}
\caption{The proposed bounds compared to the benchmark and to an Oracle, as $n$ grows and $\delta=0.05$}
\label{fig1}
\end{figure}

Next, we examine the performance of our proposed schemes for a decaying confidence level, $\delta=1/n^2$. As above, we set $A=100$ and focus on the two benchmark distributions. Figure \ref{fig2} demonstrates the results we achieve. Here, we see the advantage of Theorem \ref{thm:aryeh-ub-emp}, as it outperforms the alternatives for relatively large $n$. Once again, all bounds attain the prescribed confidence level. 
\begin{figure}[ht]
\centering
\includegraphics[width =0.8\textwidth,bb=35 210 600 570,clip]{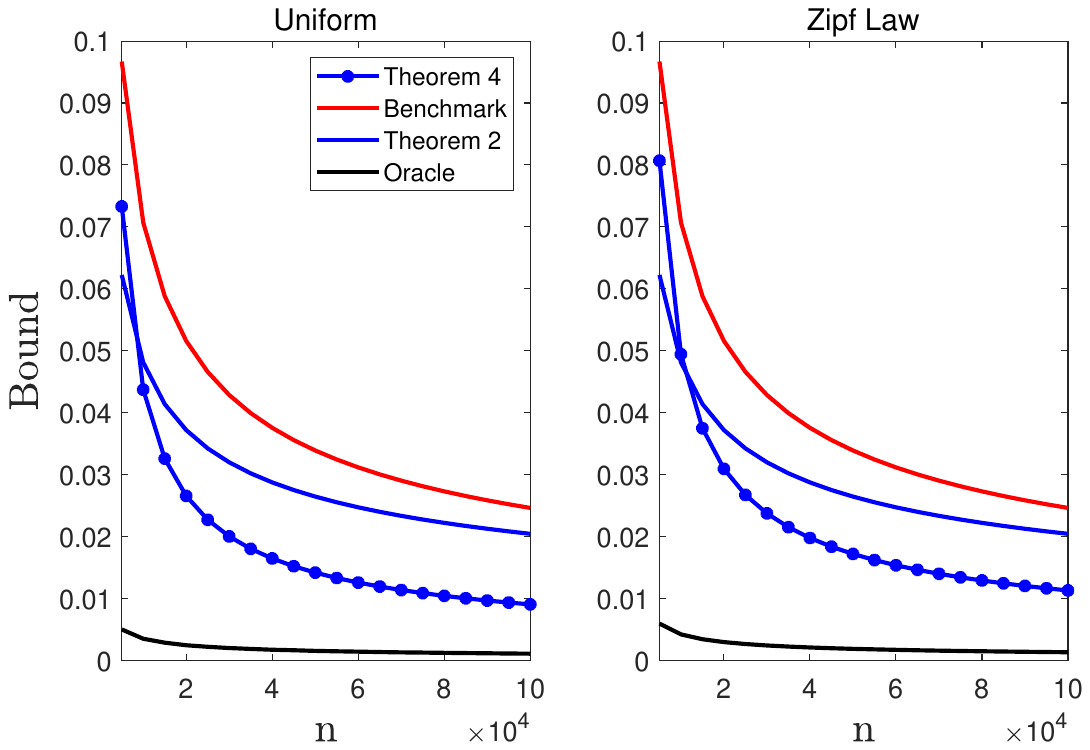}
\caption{The proposed bounds compared to the benchmark and to an Oracle, as $n$ grows and $\delta=1/n^2$}
\label{fig2}
\end{figure}

\section{Application to Inference of Frequent Events} \label{Application}

We now introduce an important application of our proposed scheme. Consider a survey asking individuals for their favorite food. We would like to report the $k$ most popular foods along with their associated CIs. The common approach is to construct $k$ marginal (binomial) intervals of confidence level $1-\delta$ each. This approach is genuinely wrong. For example, consider the case of $k=1$, $n=100$ and a uniform distribution over an alphabet size $A=||p||_0$. By definition, the most popular food in the sample would attain at least a single vote. Therefore, its exact lower bound CI (of level $1-\delta=0.95$) is at least $\theta_l=2.5\cdot10^{-4}$. This means that for $A>4000$, we attain zero coverage rate (!). This phenomenon is not quite surprising. Traditional (frequentist) inference assumes a fixed and unknown parameter $\theta$. Here, the inferred parameter is data-dependent, as it corresponds to the most frequent symbols in the sample. That is, we may obtain different $k$ most popular foods for different samples. This type of inference problem is known as \textit{selective inference} \citep{ben2017concentration}. Selective inference is a complicated task which is extensively studied in recent years \citep{tibshirani2016exact, berk2013valid,painsky2024confidence}. One of the first major contributions to the problem is due to \citet{benjamini2005false}. In their work, they showed that conditional coverage, following any selection rule for any set of (unknown) values for the parameters, is impossible to achieve. This means we cannot simply infer on the chosen parameters, given that they were selected.

Naturally, by controlling the infinity norm we implicitly control of the $k$ most frequent events. That is,  assume we found $T_\delta(X^n)$ that satisfies (\ref{infinity norm}).  Then, we have $|p_i-\hat{p}_i(X^n)|\geq T_\delta(X^n)$ for every $i \in \mathbb{N}$ with probability $1-\delta$, including the $k$ most frequent events, $\hat{p}_{[i]}, i=1,...,k$. However, it is of a natural concern that such an approach is not tight enough, as it is oblivious to $k$. In the following we study this claim and discuss the tightness of the infinity norm with respect to the $k$ most frequent events.

\begin{theorem}\label{selective_inference_LB}
Let $p=p_{i\in \mathbb{N}}$ be a distribution over $\mathbb{N}$. Let $\hat{p}$ be the MLE of $p$.
     Let $j=\text{argmax}\; \hat{p}_i$ be the most frequent symbol in the sample. Assume there exists $U_\delta(X^n)$ such that
    \begin{align}\label{most_frequent}
        \P\left(|p_j-\hat{p}_j|\geq U_\delta(X^n)  \right) \leq \delta.
    \end{align}
    Then, 
    \begin{align}\label{min_length}
  \mathbb{E}(U_\delta(X^n))\geq z_{\delta/2}\sqrt{\frac{p_{[1]}(1-p_{[1]})}{n}}+O\left(\frac{1}{n}\right)      
    \end{align}
for sufficiently large $n$.
\end{theorem}

The proof of Theorem \ref{selective_inference_LB} is provided in Section \ref{proof8}. It  utilizes the optimally of the CP CI and additional asymptotic properties. Let us now consider the $k$ most frequent events. We would like to refrain from multiplicity corrections, so we seek an interval for $\sup_{i \in \mathcal{X}_k(X^n)} |p_i-\hat{p}_i(X^n)|$ where $\mathcal{X}_k(X^n)$ is the collection of the $k$ most frequent events in the sample. This set naturally contains the single most frequent event, so a CI of average length (\ref{min_length}) is inevitable. 

Let us compare Theorem \ref{selective_inference_LB} with our proposed bounds. For this purpose we turn to real-world data sets. Notice that in the real-world settings, the true underlying probability is unknown. Hence, we treat the empirical distribution of the full data-set as the underlying distribution and sample from it accordingly. We begin with a census data; we consider the $2000$ United States Census \citep{us2014frequently}, which lists the frequency of the top $1000$ most common last names in the United States. We randomly sample $n$ names (with replacement) and examine the studied bounds for $\delta=0.05$. In addition, we present the Oracle CIs for the single most frequent symbol and the infinity norm. The left chart of Figure \ref{fig3} demonstrates the results we achieve. As we can see, the Oracle CIs are very close to each other and the difference between them and Theorem \ref{selective_inference_LB} is also negligible. This shows that the infinity norm is a very good proxy to the $k$ most frequent symbols in the alphabet. As we further examine our results, we see that for a typical experiment of $n=10000$, the top $k=5$ surnames are Smith, Johnson,
Williams, Brown and Jones with $\hat{p}_i=[0.0213,0.0165,0.0137,0.0134,0.0130]$ respectively. Theorem $\ref{main_result_1}$ attains a bound of $0.0108$ while the benchmark is about three times greater, $0.0345$. Next, we consider a corpus linguistic experiment. The popular Broadway play \textit{Hamilton} consists of $20{,}520$ words, of which $3{,}578$ are distinct. We randomly sample $n$ words (with replacement), and evaluate the corresponding bounds. The right chart of Figure \ref{fig3} demonstrates the  results we achieve. Once again,it is quite evident that Theorem \ref{main_result_1} outperforms its alternatives in this fixed $\delta=0.05$ regime. Further, we observe that the the infinity norm is a tight proxy to the $k$ most frequent symbols in the alphabet.
\begin{figure}[ht]
\centering
\includegraphics[width =0.75\textwidth,bb=20 190 600 590,clip]{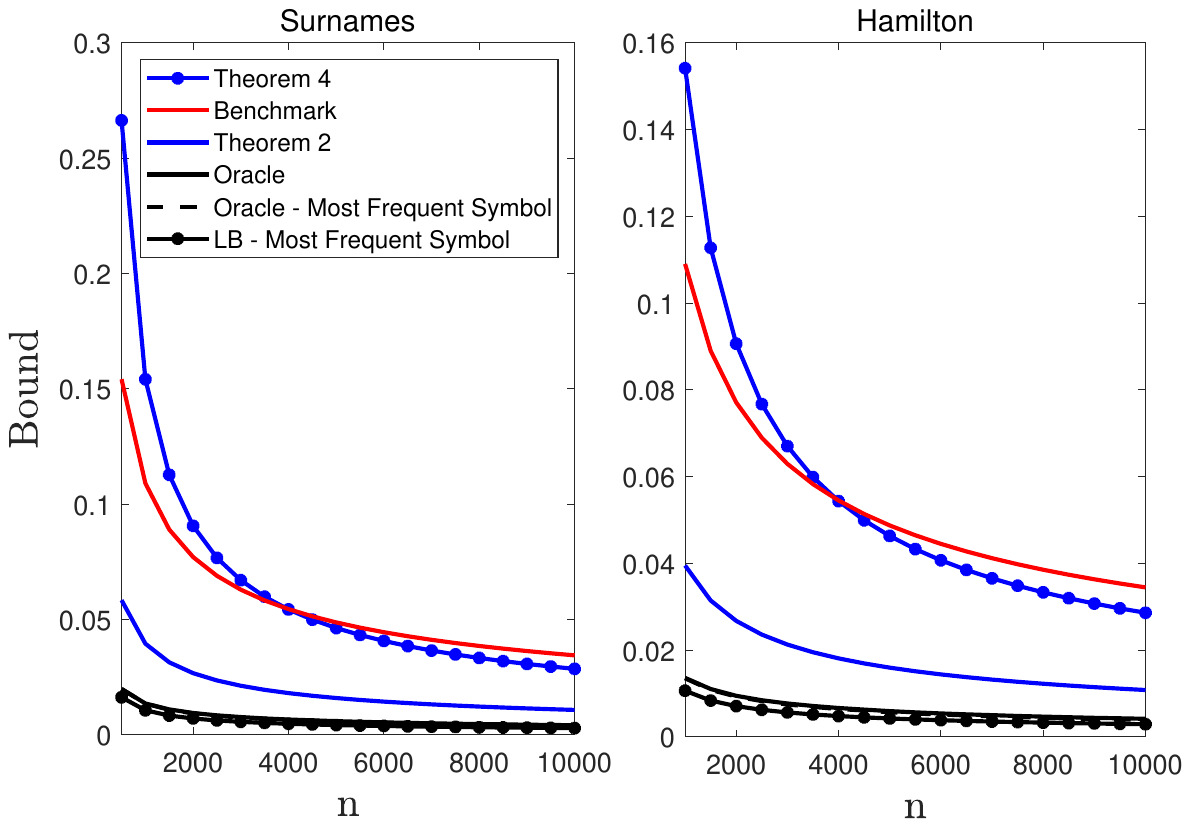}
\caption{The proposed bounds compared to the benchmark and to an Oracle, as $n$ grows. The lower bound for the most frequent symbol corresponds to Theorem \ref{selective_inference_LB}}
\label{fig3}
\end{figure}

\section{Discussion and Conclusions}
In this work we study distribution estimation under the $\ell_\infty$ norm. We introduce two data-dependent upper bounds for the MLE, which significantly improve upon currently known results.   
Our first bound (Theorem \ref{main_result_1}) demonstrates favorable performance in small sample size and fixed $\delta$ regimes. However, its dependency in the data is somewhat involved, compared to our ``dream'' result (\ref{eq:dream}). Our second bound (Theorem \ref{thm:aryeh-ub-emp}) improves the explicit dependency in the data, and demonstrates favorable performance in larger sample regimes, where $\delta$ decays with $n$. 
The above upper bounds are matched by nearly-optimal lower bounds, demonstrating the tightness of our analysis. Finally, we introduce an important application to our work in selective inference. We show that by utilizing  $\ell_\infty$ results, we provide relatively tight confidence interval for the most frequent events in the sample.

\section*{Acknowledgments}
The authors thank
Yanjun Han
and 
V\'{a}clav Vor\'{a}\v{c}ek
for enlightening discussions.
AK is partially supported by the Israel Science Foundation (grant No. 1602/19), an Amazon Research Award, and the Ben-Gurion University Data Science Research Center.
AP is partially supported by the Israel Science Foundation (grant No. 963/21).

\section{Proofs}

\subsection{A Proof for Theorem \ref{T_data_independent}}\label{proof1}
First, notice we have \begin{align}\label{main_data_depented_tight}\nonumber
    \mathbb{E}\big(\sup_i |p_i-\hat{p}_i(X^n)|\big)^m\overset{(\text{i})}{=}&\mathbb{E}\left(\sup_i (p_i-\hat{p}_i(X^n))^m\right)\overset{(\text{ii})}{\leq}\mathbb{E}\left(\sum_i (p_i-\hat{p}_i(X^n))^m\right)=\\\nonumber
    &\frac{1}{n^m}\sum_i\mathbb{E}(n_i-np_i)^m\overset{(\text{iii})}{\leq}\frac{1}{n^m} \sum_i \sum_{k=1}^d k^{m-k}(np_i(1-p_i))^k
    \end{align}
    where $d=n/2$ and
    \begin{enumerate}[(i)]
    \item follows  from the monotonicity of the power function.
    \item The supremum of non-negative elements is bounded from above by their sum \citep{maddox1988elements}.   
    \item follows from Theorem $4$ of \cite{skorski2020handy} 
    \end{enumerate}
    Applying Markov's inequality we obtain
    \begin{align}
    \P\big(\sup_i |p_i-\hat{p}_i(X^n)|\geq a\big)\leq &\frac{1}{a^m}\mathbb{E}\left(\sup_i |p_i-\hat{p}_i(X^n)|\right)^m\leq\\\nonumber
    &\frac{1}{a^m}\frac{1}{n^m} \sum_i \sum_{k=1}^d k^{m-k}(np_i(1-p_i))^k.
    \end{align}
    Setting the right hand side to equal $\delta$ yields 
    $$a=\frac{1}{n}\left(\frac{1}{\delta_1} \sum_i \sum_{k=1}^d k^{m-k}(np_i(1-p_i))^k\right)^{1/m}. $$
Therefore, with probability $1-\delta$, we have
\begin{align}\label{ba}
    \sup_i |p_i-\hat{p}_i&(X^n)|\leq \frac{1}{n}\left(\frac{1}{\delta}  \sum_{k=1}^d k^{m-k}n^k\sum_ip_i^k(1-p_i)^k\right)^{1/m}.
    \end{align}
Let us further bound from above the right hand side of (\ref{ba}). We have,  
\begin{align}\label{B2}
    \sum_ip_i^k(1-p_i)^k\overset{(\text{i})}{\leq}& \max_{t\in [0,1]} t^{k-1}(1-t)^{k}\overset{(\text{ii})}{\leq}\left(\frac{k-1}{2k-1}\right)^{k-1}\left(\frac{k}{2k-1}\right)^{k} \overset{(\text{iii})}{\leq}\\\nonumber
    &\left(\frac{k}{2k-1}\right)^{2k-1}=\left(1+\frac{-k+1}{2k-1}\right)^{2k-1}\overset{(\text{iv})}{\leq} \exp(-k+1)
\end{align}
\begin{enumerate}[(i)]
\item follows from $\sum_i p_i \psi(p_i) \leq \max_{t\in[0,1]}\psi(t)$ for $\psi(p_i)=p_i^{k-1}(1-p_i)^{k}$. That is, the mean of a random variable not greater than its maximum. 
\item simple derivation shows that the maximum of $t^{k-1}(1-t)^k$ is attained for $t^*=\frac{k-1}{2k-1}$.    
\item is due to $\left(\frac{k-1}{2k-1}\right)^{k-1}\leq \left(\frac{k}{2k-1}\right)^{k-1}$. 
\item follows from Bernoulli inequality. 
\end{enumerate}
Importantly, notice that for a choice of $p=[1/2,1/2,0,\dots,0]$ we have
\begin{align}
    \sum_ip_i^k(1-p_i)^k=2\left(\frac{1}{2}\right)^{2k}=\left(\frac{1}{2}\right)^{2k-1},
\end{align}
which approaches the term on the right hand side of inequality (iii), as $k$ increases. 
Finally, plugging (\ref{B2}) to (\ref{ba}) we obtain 
\begin{align}\label{ba2}
    \sup_i |p_i-\hat{p}_i(X^n)|\leq& \frac{1}{n}\left(\frac{1}{\delta}  \sum_{k=1}^d k^{m-k}n^k\exp(-k+1)\right)^{1/m} =\\\nonumber
            &\frac{1}{\sqrt{n}}\left(\frac{\sqrt{m/2}}{\delta^{1/m}}\right)\exp\left(-\frac{1}{2}+\frac{1}{m}\right)+O\left(\frac{1}{n^{\frac{1}{2}+\frac{1}{m}}}\right)
\end{align}
for every even $m>0$.

\subsection{A Proof for Theorem \ref{main_result_1}}\label{proof2}
We begin with the following proposition. 
\begin{proposition}\label{P2_tight}
        Let $\delta_2>0$. Then, with probability $1-\delta_2$, 
\begin{align}
    \sum_i\sum_{k=1}^{m/2} k^{m-k} (&np_i(1-p_i))^k\leq \frac{n}{n-1}\bigg( \sum_i\sum_{k=1}^{m/2} 
  k^{m-k}(n\hat{p}_i(1-\hat{p}_i))^k+\epsilon \bigg)
\end{align}
for every even $m$, where 
\begin{align}
    \epsilon=\sqrt{\frac{n}{2}\log(1/\delta_2)}\sum_{k=1}^d &k^{m-k}n^k  \bigg(\frac{k}{n 4^{k-1}}+\frac{3k(k-1)(k-2)}{n^3\cdot 2^{2k-5}} \bigg)
    .
\end{align}
\end{proposition}
\begin{proof}
    Define $\psi(n,d,\hat{p})=\sum_i\sum_{k=1}^dk^{m-k}(n\hat{p}_i(1-\hat{p}_i))^k$.
    McDiarmind's inequality suggests that
    \begin{align}\nonumber   
    \text{P}\left(\psi(n,d,\hat{p})-\mathbb{E}\left(\psi(n,d,\hat{p})\right)\leq -\epsilon  \right)\leq \exp\left(\frac{-2\epsilon^2}{\sum_{j=1}^n c_j^2} \right)
    \end{align}
    where
    \begin{align}
        \sup_{x'_j \in \mathcal{X}}\big|\psi(n,d,\hat{p})-\psi(n,d,\hat{p}')\big|\leq c_j.
    \end{align}
    where $\hat{p}'$ is the MLE over the same sample $x^n$, but with a different $j^{th}$ observation, $x_j'$.
    First, let us find $c_j$. We have
    \begin{align}\label{c_j_tight}
    &\sup_{x'_j \in \mathcal{X}}\big|\psi(n,d,\hat{p})-\psi(n,d,\hat{p}')\big|\overset{(\text{i})}{\leq}\\\nonumber 
    &\sup_{p\in [0,1-1/n]}2\big|\sum_{k=1}^dk^{m-k}\left(np(1-p)\right)^k-\sum_{k=1}^dk^{m-k}\left(n(p+1/n)(1-(p+1/n))\right)^k)\big|=\\\nonumber
    &\sup_{p\in [0,1-1/n]}2\big|\sum_{k=1}^d k^{m-k} n^k \left(p(1-p)\right)^k-\left((p+1/n)(1-(p+1/n))\right)^k)\big|\overset{(\text{ii})}{\leq}\\\nonumber
    &2\sum_{k=1}^d k^{m-k} n^k \left(\frac{k}{n\cdot 4^{k-1}}+\frac{3k(k-1)(k-2)}{n^3\cdot 2^{2k-5}}\right)
    \end{align}
where 
\begin{enumerate}
    \item Changing a single observation effects only two symbols (for example, $\hat{p}_l$ and $\hat{p}_t$), where the change is $\pm 1/n$. 
    \item Please refer to Appendix A below. 
\end{enumerate}
Next, we have
\begin{align}\nonumber
\mathbb{E}(\psi(n,d,\hat{p}))\geq& \sum_i\sum_{k=1}^d k^{m-k} n^k\left(\mathbb{E}(\hat{p}_i(1-\hat{p}_i)) \right)^k=\\\nonumber
&\sum_i\sum_{k=1}^d k^{m-k} n^k \left(\left(1-\frac{1}{n}\right)p_i(1-p_i)\right)^k\geq\\
&\left(1-\frac{1}{n}\right) \sum_i\sum_{k=1}^d k^{m-k} (np_i(1-p_i))^k\label{V1_tight}
\end{align}
where the first inequality follows from Jensen Inequality and the equality that follows is due to $\mathbb{E}(\hat{p}_i(1-\hat{p}_i))=p(1-p)(1-1/n)$. Going back to McDiarmind's inequality, we have 
\begin{align}\label{ineq1} \P \left(\mathbb{E}\psi(n,d,\hat{p})\geq \psi(n,d,\hat{p})+\epsilon  \right)\leq \exp\left(\frac{-2\epsilon^2}{n c_j^2} \right)
\end{align} 
In word, the probability that the random variable $Z=\psi(n,d,\hat{p})$ is smaller than a constant $C=\mathbb{E}(\psi(n,d,\hat{p}))-\epsilon$ is not greater that $\nu=\exp\left(-2\epsilon^2/{\sum_{j=1}^n c_j^2} \right)$. Therefore, it necessarily means that the probability that $Z$ is smaller than a constant smaller than $C$, is also not greater than $\nu$. Hence, plugging  (\ref{V1_tight})  we obtain
\begin{align}\nonumber
\P\left(\left(1-\frac{1}{n}\right) \psi(n,d,p)\geq \psi(n,d,\hat{p})+\epsilon  \right)\leq \exp\left(\frac{-2\epsilon^2}{\sum_j c_j^2} \right)
\end{align}
Setting the right hand side to equal $\delta_2$ we get 
\begin{align}
\epsilon=\sqrt{\frac{n}{2}\log(1/\delta_2)}\sum_{k=1}^d& k^{m-k} n^k\left( \frac{k}{n\cdot 4^{k-1}}+\frac{3k(k-1)(k-2)}{n^3\cdot 2^{2k-5}}\right)
\end{align} and with probability $1-\delta_2$, 
\begin{align}
    \sum_i\sum_{k=1}^d k^{m-k}& (np_i(1-p_i))^k\leq \frac{n}{n-1}\left( \sum_i\sum_{k=1}^d 
  k^{m-k}(n\hat{p}_i(1-\hat{p}_i))^k+\epsilon \right)
\end{align}
\end{proof}

Finally, we apply the union bound to (\ref{ba}) with $\delta=\delta_1$ and Proposition  \ref{P2_tight} to obtain the stated result.

\subsubsection{A Proof for Corollary \ref{T2_data_dependent}}\label{proof2.1}
We prove the Corollary  with two propositions. 
    \begin{proposition}\label{P1_new}
        Let $\delta_1>0$. Then, with probability $1-\delta_1$, 
        \begin{align}
            \sup_{i  \in \mathcal{X}} |p_i-\hat{p}_i(X^n)|\leq \frac{m}{2n}\left(\frac{1}{\delta_1}\right)^{1/m}\left(\sum_i\sum_{k=1}^{m/2}(np_i(1-p_i))^k\right)^{1/m}
        \end{align}
        for every even $m>0$.
    \end{proposition}

    \begin{proof}
    First, we have
    \begin{align}\label{main_data_depented}\nonumber
    \mathbb{E}\big(\sup_i |p_i-\hat{p}_i(X^n)|\big)^m\overset{(\text{i})}{\leq} &\frac{1}{n^m} \sum_i \sum_{k=1}^d k^{m-k}(np_i(1-p_i))^k\overset{(\text{ii})}{\leq}\\\nonumber
    &\left(\frac{d}{m}\right)^m \sum_i\sum_{k=1}^d (np_i(1-p_i))^k%
    \end{align}
    where $d=n/2$ and
    \begin{enumerate}
    \item follows  from (15) in the main text .
     \item follows from $k^{m-k}\leq  d^m$ for every $k\in\{1,...,d\}$.  
    \end{enumerate}
    Applying Markov's inequality we obtain
    \begin{align}
    \P\left(\sup_i |p_i-\hat{p}_i(X^n)|\geq a\right)\leq &\frac{1}{a^m}\mathbb{E}\left(\sup_i |p_i-\hat{p}_i(X^n)|\right)^m\leq\\\nonumber
    &\frac{1}{a^m}\left(\frac{d}{n}\right)^m \sum_i\sum_{k=1}^d (np_i(1-p_i))^k.
    \end{align}
    Setting the right hand side to equal $\delta_1$ yields 
    $$a=\left(\frac{1}{\delta_1}\left(\frac{d}{n}\right)^m \sum_i\sum_{k=1}^d (np_i(1-p_i))^k\right)^{1/m}=\frac{m}{2n} \left(\frac{1}{\delta_1}\sum_i\sum_{k=1}^{m/2} (np_i(1-p_i))^k\right)^{1/m}. $$
    \end{proof}

\begin{proposition}\label{P2_new}
        Let $\delta_2>0$. Then, with probability $1-\delta_2$, 
        
\begin{align}
    &\sum_i\sum_{k=1}^d (np_i(1-p_i))^k\leq\\\nonumber
    &\frac{n}{n-1}\left( \sum_i\sum_{k=1}^d(n\hat{p}_i(1-\hat{p}_i))^k+d\sqrt{\frac{1}{2}\log(1/\delta_2)}\left(2n^{d-1/2}+48n^{d-5/2}\right)\right)
\end{align}
for every even $m$. 
\end{proposition}
\begin{proof}
    McDiarmind's inequality suggests that
    \begin{align}\nonumber   
    \P\left(\sum_i\sum_{k=1}^d(n\hat{p}_i(1-\hat{p}_i))^k-\mathbb{E}\left(\sum_i\sum_{k=1}^d(n\hat{p}_i(1-\hat{p}_i))^k\right)\leq -\epsilon  \right)\leq \exp\left(\frac{-2\epsilon^2}{\sum_{j=1}^n c_j^2} \right)
    \end{align}
    where
    \begin{align}
        \sup_{x'_j \in \mathcal{X}}\big|\sum_i\sum_{k=1}^d\left(n\hat{p}_i(1-\hat{p}_i)\right)^k-\sum_i\sum_{k=1}^d\left(n\hat{p}'_i(1-\hat{p}'_i)\right)^k)\big|\leq c_j.
    \end{align}
    First, let us find $c_j$. We have
    \begin{align}\label{c_j}
    &\sup_{x'_j \in \mathcal{X}}\big|\sum_i\sum_{k=1}^d\left(n\hat{p}_i(1-\hat{p}_i)\right)^k-\sum_i\sum_{k=1}^d\left(n\hat{p}'_i(1-\hat{p}'_i)\right)^k)\big|\overset{(\text{i})}{\leq}\\\nonumber &2\sup_{p\in [0,1-1/n]}\big|\sum_{k=1}^d\left(np(1-p)\right)^k-\sum_{k=1}^d\left(n(p+1/n)(1-(p+1/n))\right)^k)\big|=\\\nonumber
    &2\sup_{p\in [0,1-1/n]}\big|\sum_{k=1}^d n^k \left(p(1-p)\right)^k-\left((p+1/n)(1-(p+1/n))\right)^k)\big|\leq\\\nonumber
    &2\sum_{k=1}^dn^k\sup_{p\in [0,1-1/n]}\big|\left(p(1-p)\right)^k-\left((p+1/n)(1-(p+1/n))\right)^k)\big|\overset{(\text{ii})}{=}\\\nonumber
    &2\sum_{k=1}^d n^k \left( \frac{k}{n\cdot 4^{k-1}}+\frac{3k(k-1)(k-2)}{n^3\cdot 2^{2k-5}}\right)\leq\\\nonumber
    &n^{d-1}\sum_{k=1}^d\left(\frac{2k}{4^{k-1}}+\frac{3k(k-1)(k-2)}{n^2\cdot 2^{2k-4}}\right)\overset{(\text{iii})}{\leq} 2dn^{d-1}+48dn^{d-3}
    \end{align}
where 
\begin{enumerate}
    \item Changing a single observation effects only two symbols (for example, $\hat{p}_l$ and $\hat{p}_t$), where the change is $\pm 1/n$. 
    \item Please refer to Appendix A. \item Follows from  $\sum_{k=1}^d\frac{k}{4^{k-1}}=4\sum_{k=1}^d\frac{k}{4^{k}}\leq d$ and 
    \begin{align}
    \sum_{k=1}^d\frac{k(k-1)(k-2)}{4^{k-2}}\leq \sum_{k=1}^d\frac{k^3}{4^{k-2}}\leq  d \max_{k\in [1,d]}\frac{k^3}{4^{k-2}}\leq 16\frac{2\exp(-3)}{\log(4)}\leq 16
    \end{align}
     where the maximum is obtain for $k^*=3/\log(4)$.
\end{enumerate}

next, we have
\begin{align}\label{V1}
\mathbb{E}\bigg(\sum_i\sum_{k=1}^d&(n\hat{p}_i(1-\hat{p}_i))^k\bigg)\geq \sum_i\sum_{k=1}^d \left(\mathbb{E}(n\hat{p}_i(1-\hat{p}_i)) \right)^k=\\\nonumber
&\sum_i\sum_{k=1}^d n^k \left(\left(1-\frac{1}{n}\right)p_i(1-p_i)\right)^k\geq\left(1-\frac{1}{n}\right) \sum_i\sum_{k=1}^d (np_i(1-p_i))^k
\end{align}
Going back to McDiarmind's inequality, we have 
\begin{align}    \P\left(\mathbb{E}\left(\sum_i\sum_{k=1}^d(n\hat{p}_i(1-\hat{p}_i))^k\right)\geq \sum_i\sum_{k=1}^d(n\hat{p}_i(1-\hat{p}_i))^k+\epsilon  \right)\leq \exp\left(\frac{-2\epsilon^2}{\sum_{j=1}^n c_j^2} \right)
\end{align} 
Plugging  (\ref{V1})  we obtain
\begin{align}\nonumber
\P\left(\left(1-\frac{1}{n}\right) \sum_i\sum_{k=1}^d (np_i(1-p_i))^k\geq \sum_i\sum_{k=1}^d(n\hat{p}_i(1-\hat{p}_i))^k+\epsilon  \right)\leq \exp\left(\frac{-2\epsilon^2}{\sum_j c_j^2} \right)
\end{align}
Setting the right hand side to equal $\delta_2$ we get $$\epsilon=\sqrt{\frac{n}{2}\log(1/\delta_2)}\left(2dn^{d-1}+48dn^{d-3}\right)$$ and with probability $1-\delta_2$, 
\begin{align}
    &\sum_i\sum_{k=1}^d (np_i(1-p_i))^k\leq\\\nonumber
    &\frac{n}{n-1}\left( \sum_i\sum_{k=1}^d(n\hat{p}_i(1-\hat{p}_i))^k+d\sqrt{\frac{1}{2}\log(1/\delta_2)}\left(2n^{d-1/2}+48n^{d-5/2}\right)\right)
\end{align}

\end{proof}

Finally, we apply the union bound to Propositions \ref{P1_new} and \ref{P2_new} to obtain 
\begin{align}\nonumber
&\sup_{i  \in \mathcal{X}} |p_i-\hat{p}_i(X^n)|\leq \\\nonumber
&\frac{m}{2n}\left( \frac{1}{\delta_1} \frac{n}{n-1}\left(\sum_i\sum_{k=1}^{d}(n\hat{p}_i(1-\hat{p}_i))^k+d\sqrt{\frac{1}{2}\log(1/\delta_2)}\left(2n^{d-1/2}+48n^{d-5/2}\right)\right)  \right)^{1/m}\leq\\\nonumber
&\frac{m}{2\delta_1^{1/m}}\frac{1}{n}\left(\frac{n}{n-1}\right)^{1/m}\left( \sum_i\sum_{k=1}^{m/2}(n\hat{p}_i(1-\hat{p}_i))^k \right)^{1/m}+\\\nonumber
&\frac{m}{2\delta_1^{1/m}}\frac{1}{n}\left(\frac{n}{n-1}\right)^{1/m}(m/2)^{1/m}\left(\frac{1}{2}\log\left(\frac{1}{\delta_2}\right)\right)^{1/m}\left(2n^{\frac{1}{2}-\frac{1}{2m}}+48n^{\frac{1}{2}-\frac{5}{2m}}\right)
\end{align}
with probability $1-\delta_1-\delta_2$. Define $g(m,\delta_1)={m}/{\delta_1^{1/m}}$. Further, it is immediate to show that $(m/2)^{1/m}\leq \sqrt{\exp(1/\exp(1))}$. Hence, with probability $1-\delta_1-\delta_2$,
\begin{align}\nonumber
\sup_{i  \in \mathcal{X}} |p_i-\hat{p}_i(X^n)|\leq& \frac{g(m,\delta_1)}{n}\left( \sum_i\sum_{k=1}^{m/2}(n\hat{p}_i(1-\hat{p}_i))^k \right)^{1/m}+\\\nonumber
&bg(m,\delta_1) (\log(1/\delta_2))^{1/2m} \left( n^{-\frac{1}{2}\left(1+\frac{1}{m}\right)}+24n^{-\frac{1}{2}\left(1+\frac{5}{m}\right)}\right)
\end{align}
for every even $m$, where  $b=\sqrt{2\exp(1/\exp(1))}$. Finally, we would like to choose $m$ which minimizes $g(m,\delta_1)$. We show in Appendix B that $\inf_m g(m,\delta_1)=\exp(1) \log(1/\delta_1)$, where and the infimum is obtained for a choice of $m^*=\log(1/\delta_1)$. 

\subsection{A Proof of Theorem \ref{thm:aryeh-ub}}\label{proof3}
Let us first introduce some auxiliary results and background
\subsubsection{Auxiliary Results}
\newcommand{\pl}[1]{\sqprn{#1}\subplus}
\newcommand{\dspl}[1]{\dsparen{#1}\subplus}
\newcommand{\subplus}{_{
\scalebox{.5}{
\!\!\!\!\!
$\boldsymbol{+}$}
}}
\begin{lemma}[contained in the proof of
Lemma 10,
\cite{CohenK23}]
\label{lem:CK}
Let
$Y_{i\in I \subseteq \N}$
be
random variables
such that, for each $i \in I$, there 
are $v_i>0$ and $a_i \geq 0$
satisfying
\beqn
\label{eq:bern-cond}
\PR{Y_i \geq \eps}
& \leq&
\exp\paren {- \frac{\eps^2}{2(v_i + a_i \eps)}},
\qquad\eps\ge0
.
\eeqn
Put
\beqn
\begin{array}{l}
v^*:=
\sup_{i \in I}v_i,
\quad
V^*:=
\sup_{i \in I}
 v_i \log (i+1),
\quad
a^*:=
\sup_{i \in I}a_i,
\quad
A^*:=
\sup_{i \in I} a_i \log (i+1).
\end{array}
\eeqn

Then
\beq
\PR{
\sup_{i \in I}Y_i
\ge
2\sqrt{V^*+v^*\log\frac1\delta}
+4A^*
+4a^*\log\frac1\delta
}
&\le&\delta
.
\eeq
\end{lemma}

\begin{remark}
\label{rem:p1/2}
When considering the random variable
$Z=\sup_{i \in \N}|\hat p_i-p_i|$,
there is no loss of generality in assuming
that $p_i\le1/2$, $i\in\N$.
Indeed, $|Y_i|=|\hat p_i-p_i|$
is distributed as 
$|n\inv\Bin(n,p_i)-p_i|$,
and the latter distribution is invariant under
the transformation $p_i\mapsto1-p_i$.
\end{remark}

\begin{lemma}
\label{lem:vsubgauss}
For any distribution $p_{i\in\N}$,
\beq
V(p)
&\le &
\phi(v^*(p))
.
\eeq
\end{lemma}
\begin{proof}
(This elegant proof idea is due to
V\'aclav Vor\'a\v{c}ek.)
There is no loss of generality
in assuming $p=p^\downarrow$.
The claim then amounts to
\beq
\sup_{i\in\N}
v_i\log(i+1)
\le 
v^*\log\frac{1}{v^*}
.
\eeq
The monotonicity of the $p_i$
implies 
$
p_i
\le
(p_1+\ldots+p_i)/i
\le1/i
$.
Now
$x\le1/i\implies
x(1-x)\le1/(i+1)
$ for $i\in\N$,
and hence
$v_i\le1/(i+1)$.
Thus,
$v_i\log(i+1)\le 
v_i\log\frac{1}{v_i}
$.
Finally, since
$x\log(1/x)$ is increasing
on $[0,1/4]$,
which is the range of the $v_i$,
we have
$
\sup_{i\in\N}
v_i\log\frac{1}{v_i}
\le
v^*\log\frac{1}{v^*}
$.
\end{proof}
\begin{remark}
    \label{rem:vv}
There is no reverse inequality
of the form
$\phi(v^*(p))\le F(V^*(p))$,
for any fixed $F:\R_+\to\R_+$.
This can be seen by considering $p$ supported on $[k]$,
with $p_1=\log(k)/k$ and the remaining masses
uniform. Then $V^*(p)\approx \log(k)/k$
while 
$\phi(v^*(p))
\approx
\log(k)\log(k/\log k)/k
$.
\end{remark}

\begin{proposition}
\label{prop:betan}
Let $n\geq10$ and $\beta=\log(n)$. Then, 
$$f(n)=\frac{\beta^{-\beta} n^2 \left(\frac{n-\beta}{n}\right)^{\beta-n}}{2^\beta-2}\leq
\frac{81}{2}.$$
\end{proposition}
\begin{proof}
    To prove the above, we show that $f(n)$ is decreasing for $n>200$. This means that the maximum of $f(n)$ may be numerically evaluated in the range $n\in\{10,...,200\}$. Finally, we verify that the maximum of $f(n)$ is attained for $n=33$, and is bounded from above by $81/2$ as desired. 
It remains to verify
    that $f(n)$ is decreasing for $n>200$. Since $f(n)$ is non-negative, it is enough to show that $g(n)=\log f(n)$ is decreasing. Denote
    \begin{align}
    g(n)=-\beta\log \beta+2\log n +(n-\beta)\log(n-\beta)+(n-\beta)\log n -\log(2^\beta-2).
    \end{align}
    Taking the derivative of $g(n)$ we have,
    \begin{align}
        &g'(n)=\\\nonumber
        &-\frac{1}{n}(\log \beta+1)+\frac{2}{n}+\left(1-\frac{1}{n}\right)(-\log(n-\beta)-1+\log n) +\frac{n-\beta}{n}-\frac{1}{n}\frac{2^\beta\log 2}{2^\beta-2}=\\\nonumber
        &\frac{1}{n}\left((n-1)\log \frac{n}{n-\beta}-\log \beta -\beta +2 -\frac{2^\beta\log 2}{2^\beta-2}\right)\leq\\\nonumber
        &\frac{1}{n}\left(n\log \frac{n}{n-\beta}-\log \beta -\beta+2 -\log 2\right)\leq
        \frac{1}{n}\left(\frac{n\beta}{n-\beta}-\log \beta -\beta +2 -\log 2\right)=\\\nonumber
        &\frac{1}{n}\left(\frac{\beta^2}{n-\beta}-\log \beta  +2 -\log 2\right),
    \end{align}
    where the first inequality follows from $\log (n /(n-\beta))\geq 1$ and $2^\beta/(2^\beta-2)\geq 1$, while the second inequality is due to Bernoulli's inequality, $(n/(n-\beta))^n\leq \exp(n\beta/(n-\beta))$.  
    Finally, it is easy to show that $\beta^2/(n-\beta)$ is decreasing for $n\geq 10$. This means that $\beta^2/(n-\beta)\leq (\log 10)^2/(10-\log(10))$ and $g'(n)<0$ for $n>200$.
\end{proof}

\begin{lemma}[generalized Fano method \citep{yu1997assouad}, Lemma 3]
\label{lem:fano}
    For 
    \(r \geq 2\), let 
    \(\mathcal{M}_r 
    \)
    be a collection of $r$
    probability measures
    \(\nu_1, \nu_2, ... ,\nu_r\)
    with some parameter of interest
    \(\theta(\nu)\) 
    taking
    values in pseudo-metric space \( (\Theta, \rho) \)
    such that for all
    \(j \neq k \), we have
    \[
    \rho(\theta(\nu_j), \theta(\nu_{k}) )
    \geq
    \alpha
    \]
    and
    \[
    D(\nu_j ~\Vert~ \nu_{k})
    \leq
    \beta.
    \]
    Then
    \[
    \inf_{\hat\theta}
    \max_{j \in [d]}
    \E_{Z \sim \mu_j}
    \rho(\hat\theta(Z), \theta(\nu_j) )
    \geq
    \frac{\alpha}{2} \paren{1 - \paren{\frac{\beta + \log 2}{\log r}}},
    \]
    where the infimum is over all estimators 
    \(\hat\theta:Z\mapsto\Theta\).
\end{lemma}

\begin{proposition}
\label{prop:pqdist2}
    Let $p$ and $q$ be two distributions with support size $n$.
Define $p$ by
\beq
p_1 = \frac{\log n}{2n\log\log n},
\quad
p_i = \frac{1-p_1}{n-1},
\qquad
i>1,
\eeq
and $q$ by $q_2=p_1$, and $q_i=p_2$ for $i\neq 2$. Then,
\begin{enumerate}
    \item $\nrm{p-q}_\infty \geq c\frac{\log n}{n\log\log n}$ for some $c>0$ and all $n$ sufficiently large. 
    \item $\lim_{n\rightarrow \infty}\frac{n}{\log n}D(p||q)=\frac{1}{2}$
\end{enumerate}
\end{proposition}
\begin{proof}
For the first part, it is enough to show that 
$$|p_1-p_2|\geq c\log(n)/n\log\log n$$
for some $c>0$ and sufficiently large $n$. 
First, we show that $p_1\geq p_2$ for $n\geq (\log n)^2$. That is,
\begin{align}
    p_1-\frac{1-p_1}{n-1}=\frac{np_1-1}{n-1}>0
\end{align}
for $np_1>1$. Next, fix $0<c \leq 1/2$. We have,
\begin{align}
    |p_1-p_2|-\frac{c\log(n)}{n\log\log n}=&\frac{ap_1-1}{n-1}-\frac{c\log n}{n\log\log n}=\\\nonumber
    &\frac{1}{n-1}\left(\frac{\log n}{2\log\log n}-1-\frac{n-1}{n}\frac{c\log n}{\log\log n}  \right)=\\\nonumber
    &\frac{1}{(n-1)2\log\log n}\left(\log n\left(1-\frac{n-1}{n}2c\right)-2\log\log n \right)> 0
\end{align}
where the last inequality holds for $c(n-1)/n<1/2$ and sufficiently large $n$, as desired. We now proceed to the second part of the proof. 
\begin{align}
    \frac{n}{\log n}D(p||q)=\frac{n}{\log n}\left(p_1\log\frac{p_1}{q_1}+p_2\log\frac{p_2}{q_2}\right)=\frac{n}{\log n}(p_1-p_2)\log\frac{p_1}{p_2}.
\end{align}
First, we have
\begin{align}
    \frac{n}{\log n}(p_1-p_2)=&\frac{n}{\log n}\left(p_1-\frac{1-p_1}{n-1}\right)=\frac{n}{\log n}\left(\frac{np_1-1}{n-1}\right)=\\\nonumber
    &\frac{n}{\log n}\frac{\log n/2n\log\log n-1}{n-1}=\frac{n}{n-1}\left(\frac{1}{2\log \log n}-\frac{1}{\log n}\right).
\end{align}
Next, 
\begin{align}
    \log\frac{p_1}{p_2}=&\log(n-1)+
    \log\frac{p_1}{1-p_1}=\log(n-1)+\log\frac{\log n}{2n\log \log n - \log n}=\\\nonumber
    &\log(n-1)+\log \log n -2\log(2n\log\log n -\log n).
\end{align}
Putting it all together we obtain
\begin{align}
    \frac{n}{\log n}&D(p||q)=\\\nonumber
    &\frac{n}{n-1}\left(\frac{1}{2\log \log n}-\frac{1}{\log n}\right)\left( \log(n-1)+\log \log n -2\log(2n\log\log n -\log n)
    \right)= \\\nonumber
    &\frac{n}{n-1}\bigg(\frac{\log(n-1)}{2\log\log n}-\frac{\log(n-1)}{\log n}+\frac{1}{2}-\frac{\log\log n}{\log n}-\\\nonumber
    &\quad\quad\quad\;\;\frac{\log(2n\log\log n - \log n)}{2\log\log n}+\frac{\log(2n\log\log n - \log n)}{\log n}  \bigg)=\\\nonumber
    &\frac{n}{n-1}\bigg(\frac{1}{2}+\frac{\log(n-1)-\log(2n\log\log n - \log n)}{2\log\log n}+\\\nonumber
    &\quad\quad\quad\;\;\frac{\log(2n\log\log n - \log n)-\log(n-1)}{\log n}-\frac{\log\log n}{\log n}\bigg).
\end{align}
It is straightforward to show that the last three terms in the parenthesis above converge to zero for sufficiently large $n$, which leads to the stated result. 
\end{proof}

\begin{lemma}[\cite{peres-loc-coin17}]
\label{lem:bern-np}
When estimating a single Bernoulli parameter
in the range $[0,p_0]$,
$\Theta(p_0\eps^{-2}\log(1/\delta))$
draws are both necessary and sufficient 
to achieve additive accuracy $\eps$
with probability at least $1-\delta$.    
\end{lemma}

\noindent \textbf{Bernstein inequalities}\\
\noindent Background: 
Let $Y \sim \text{Bin}(n,\theta)$ be a Binomial random variable 
and let
$\hat{\theta}=Y/n$ be the 
its MLE.

\begin{itemize}
    \item Classic Bernstein \citep{boucheron2003concentration}:
\beqn
\label{eq:cla-bern}
\PR{
\hat\theta
-
\theta \ge \eps}
&\le& \exp\paren{-\frac{n\eps^2}{2(
\theta(1-\theta)
+\eps/3}}
\eeqn
with an analogous bound for the left tail.
This implies:
\beqn
\label{eq:dev-bern}
|\theta-\hat\theta|
&\le&
\sqrt{\frac{2\theta(1-\theta)}{n}\log\frac{2}{\delta}}
+
\frac{2}{3n}\log \frac{2}{\delta}
.
\eeqn

\item Empirical Bernstein \cite[Lemma 5]{DBLP:conf/icml/DasguptaH08}:
\beqn
\label{eq:emp-bern}
|\theta-\hat\theta|
&\le&
\sqrt{\frac{5\hat{\theta}(1-\hat{\theta})}{n}\log\frac{2}{\delta}}
+
\frac{5}{n}\log \frac{2}{\delta}.
\eeqn

\end{itemize}

\noindent We are now ready to present the proof of Theorem $3$.

\subsubsection{Proof of Theorem \ref{thm:aryeh-ub}}

We assume without loss of generality that 
$p$ 
is sorted in descending order: $p_1\ge p_2\ge\ldots$
and further, as per Remark~\ref{rem:p1/2}, that $p_1\le1/2$.
The estimate $\hat p_i$ is just the MLE based on $n$ iid draws.

Our strategy for analyzing $\sup_{i\in\N}\abs{\hat p_i-p_i}$ will be to break up $p$ into the ``heavy'' masses, where we
apply a maximal Bernstein-type inequality, and the ``light'' masses, where we apply a multiplicative Chernoff-type bound.

We define the 
``heavy'' masses as those with $p_i\ge1/n$.
Denote by $I\subset\N$ the set of corresponding indices
and note that
$|I|\le n$.
For $i\in I$, put $Y_i=\hat p_i-p_i$.
Then \eqref{eq:cla-bern} implies that
each $Y_i$ satisfies \eqref{eq:bern-cond}
with $v_i=p_i(1-p_i)/n$
and $a_i=1/(3n)$;
trivially,
$\max_{i\in I}a_i\log(i+1)=
\log(n+1)/(3n)
$.
Invoking Lemma~\ref{lem:CK}
twice
(once for $Y_i$ and again for $-Y_i$)
together with the union bound,

we have,
with probability $\ge1-\delta$,
\beqn
\label{eq:E-heavy}
\max_{i \in I}|\hat p_i-p_i|
\le
2\sqrt{
\frac{V^*}{n}
+
\frac{v^*}{n}\log\frac2\delta
}
+
\frac{4\log(n+1)}{3n}
+
\frac{4}{3n}
\log\frac2\delta
.
\eeqn

Next, we analyze the light masses. 
Our first ``segment'' consisted of the $p_i\in[n\inv,1]$; these were the heavy masses.
We take the next segment to consist of 
$p_i\in[(2n)\inv,n\inv]$, of which there are at most $2n$ atoms.
The segment after that will be in the range $[(4n)\inv,(2n)\inv]$, and, in general, the $k$th segment is in the range
$[(2^kn)\inv,(2^{k-1}n)\inv]$, and will contain at most $2^kn$ atoms.
To the $k$th segment, we apply the Chernoff bound $\P(\hat p\ge p+\eps)\le\exp(-n D(p+\eps||p))$,
where $p=(2^kn)\inv$ and 
$\eps=\eps_k= 2^kp
\beta
-p$,
for some $\beta$
to be specified below.
[Note that $D(\alpha p||p)$ is monotonically increasing in $p$ for fixed $\alpha$,
so we are justified in taking the left endpoint.]
For this choice, in the $k$th segment we have
\beq
D(p+\eps||p)
&=&
D(2^kp
\beta
||p)
=
D\paren{
\frac{
\beta
}n
\Big\Vert
\frac{1}{2^kn}
}
\\
&=&
\frac{
(n-
\beta
) \log \left(\frac{2^k (n-
\beta
)}{2^k n-1}\right)
+
\beta
\log \left(2^k 
\beta
\right)}{n}
\\
&\ge&
\frac{
(n-\beta) \log \left(\frac{n-\beta}{ n}\right)
+
\beta \log \left(2^k \beta\right)
}{n},
\eeq
since
neglecting the $-1/2^k$ additive term in the denominator
decreases the expression.
Let $E$ be the event that {\em any} of the $p_i$s in any of the segments $k=1,2,\ldots$ has a corresponding $\hat p_i$ that 
exceeds $\beta/n$.
Then
\beq
\P(E) &\le&
\sum_{k=1}^\infty
2^kn\exp\paren{
-
(n-\beta) \log \left(\frac{n-\beta}{ n}\right)
-
\beta \log \left(2^k \beta\right)
}
=
\frac{2 \beta^{-\beta} n \left(\frac{n-\beta}{n}\right)^{\beta-n}}{2^\beta-2}
.
\eeq

For the choice $\beta=\log n$,
we have
\beqn
\label{eq:beta-logn}
\P(E)\le
\frac{2 \beta^{-\beta} n \left(\frac{n-\beta}{n}\right)^{\beta-n}}{2^\beta-2}
&\le&
\frac{81}{n},
\qquad
n\ge10,
\eeqn
which is proved in Proposition \ref{prop:betan}.
Now $E$ is the event that
$
\sup_{i:p_i<1/n}
(\hat p_i-p_i)
\ge\log(n)/n
$.
Since $p_i<1/n$, there is no need
to consider the left-tail deviation
at this scale, as all of the probabilities
will be zero.
Combining \eqref{eq:E-heavy} with \eqref{eq:beta-logn} yields 
\eqref{eq:UB1}. Since Lemma ~\ref{lem:vsubgauss} implies that
$V^*\le\phi(v^*)$,
\eqref{eq:UB2} follows from \eqref{eq:UB1}.
Finally, \eqref{eq:UB3}
follows from \eqref{eq:UB1} via the obvious relation
$V^*\le \log(n+1)v^*$.

\subsection{A Proof for Theorem \ref{thm:aryeh-ub-emp}}\label{proof4}

We begin with an elementary observation:
for 
$N\in\N$
and
$a,b\in[0,1]^N$,
we have
\beq
\abs{
\max_{i\in[N]}
a_i(1-a_i)
-
\max_{i\in[N]}
b_i(1-b_i)
}
&\le&
\max_{i\in[N]}\abs{
a_i
-
b_i
},
\eeq
and this also carries over to
$a,b\in[0,1]^\N$.
Let us denote
$v^*:=\sup_{i\in\N}p_i(1-p_i)$
and
$\hat v^*:=\sup_{i\in\N}\hat p_i(1-\hat p_i)$. Together with \eqref{eq:UB3}, this implies
\beq
\abs{v^*-\hat v^*}
\le
\nrm{p-\hat p}_\infty
\le
a
+
b\sqrt{v^*}
\eeq
where 
\beq
a &=& \frac{4}{3n}\log\frac{2(n+1)}\delta+\frac{\log n}n,
\\
b &=&
2\sqrt{\frac{\log (n+1)}{n}+\frac{1}{n}\log\frac2\delta}
.
\eeq
Following the proof of Lemma $5$ in \cite{DBLP:conf/icml/DasguptaH08},
\beq
|v^*-\hat v^*|  &\le& a + b\sqrt{v^*}
\\
&\le& a + b\sqrt{\hat v^*+|v^*-\hat v^*|}
\\
&\le& a + b\sqrt{\hat v^*}+b\sqrt{|v^*-\hat v^*|},
\eeq
where we used $v^*\le \hat v^*+|v^*-\hat v^*|$
and $\sqrt{x+y}\le\sqrt{x}+\sqrt{y}$.
Now we have an expression of the form
\beq
A\le B\sqrt{A}+C,
\eeq
where
$A=
|v^*-\hat v^*|
$,
$B=b$,
$C=
a + b\sqrt{\hat v^*}
$, which implies
$A\le B^2+B\sqrt{C}+C$,
or
\beq
|v^*-\hat v^*|
&\le&
b^2+a + b\sqrt{\hat v^*}
+b\sqrt{a + b\sqrt{\hat v^*}}.
\eeq
Using
$\sqrt{x+y}\le\sqrt{x}+\sqrt{y}$
and
$\sqrt{xy}\le(x+y)/2$,
\beq
|v^*-\hat v^*|
&\le&
b^2+a + b\sqrt{\hat v^*}
+b\sqrt{a} + b\sqrt{b\sqrt{\hat v^*}}
\\&\le&
b^2+a + b\sqrt{\hat v^*}
+b\sqrt{a} + 
b(b+\sqrt{\hat v^*})/2
\\
&=&
a+3b^2/2
+b\sqrt{a}
+3b\sqrt{\hat v^*}/2.
\eeq
We still have
\beq
a + b\sqrt{v^*}
\le
a+3b^2/2
+b\sqrt{a}
+3b\sqrt{\hat v^*}/2,
\eeq
whence, with probability $1-\delta$,
\beqn
\nrm{p-\hat p}_\infty
\le
a+3b^2/2
+b\sqrt{a}
+3b\sqrt{\hat v^*}/2.
\eeqn

\subsection{Proof of Proposition \ref{prop:fano}}\label{proof5}

\begin{proof}

The necessity 
of an additive term of order 
$\log(n)/(n\log\log n)$
can be intuited via a balls-in-bins analysis:
If $n$ balls are uniformly thrown into $n$ bins,
we expect a maximal load of about $\log(n)/(n\log\log n)$ balls
in one of the bins \citep{DBLP:conf/random/RaabS98}.
We proceed to formalize this intuition.

Let $\mu_1,\ldots,\mu_n$ be probability
measures on $\N$ with support contained in $[n]$, defined as follows.
For $a:=
\log(n)/(2n\log\log n)
$
and $b:=(1-a)/(n-1)$,
we define,
for $i\in[n]$,
$\mu_i(i)=a$,
for $j\neq i$,
$\mu_i(j)=b$.
For each $i\in[n]$,
define the probability $\nu_i$
on $\N^n$ as the $n$-fold product measure
$\mu_i^n$. It follows from Proposition~\ref{prop:pqdist2}
that
\beq
\frac{
D(\nu_j||\nu_{k})
}{\log n}
=
\frac{
n D(\mu_j||\mu_{k})
}{\log n}
\ninf
\frac12
\eeq
and
$
\nrm{\mu_j-\mu_k}_\infty\ge\alpha:= 
c\log(n)/(n\log\log n)
$
for $j\neq k$ and $n$ sufficiently large.
Invoking Lemma~\ref{lem:fano} with 
$r=n$, 
$\theta(\nu_j) = \mu_j$, 
$\rho = \nrm{\cdot}_\infty$,
and
$\beta=\frac12\log n+o(\log n)$
completes the proof.
\end{proof}

\subsection{Proof of Proposition \ref{prop:decoup-lb}}\label{proof6}
The analysis relies on 
a result of \cite[Theorem 2]{CohenK23}.
Let $X_i\sim\Bin(n,p_i)$, $i\in\N$
be a sequence of independent 
binomials with $1/2\ge p_1\ge p_2\ge\ldots$
and define $Y_i:=n\inv X_i-p_i$.
Then \cite{CohenK23} showed\footnote{
The theorem therein claimed this for 
$\E\sup_{i\in\N}|Y_i|$
but in fact the proof shows this for
$\E\sup_{i\in\N}Y_i$.
} that
\beqn
\label{eq:ck}
c\sqrt{S} \le 
\liminf_{n\to\infty}
\sqrt n
\E\sup_{i\in\N}Y_i
,
\eeqn
where 
$
S:=\sup_{i\in\N}p_i\log(i+1)
$
and
$c>0$ is an absolute constant.

Since 
$
t\le 2t(1-t)
$
for
$t\in[0,1/2]$
and $p_1\le1/2$
(as per Remark~\ref{rem:p1/2}),
we have that $V^*\ge S/2$.
However, 
the
\cite{CohenK23} 
lower bound
is not immediately
applicable to our case, because 
\eqref{eq:ck} 
requires the binomials to be independent.
Fortunately, their dependence is of the 
{\em negative association}
type
\cite[Theorem 14]{Dubhashi:1998:BBS:299633.299634},
which futher implies
negative right orthant dependence
(Proposition 5, ibid.).
Finally,
\cite[Proposition 4]{kontorovich2023decoupling}
shows that
\beqn
\label{eq:kon-decoup}
\E\sup_{i\in\N}Y_i
&\ge&
\frac12\E\sup_{i\in\N}\tilde Y_i,
\eeqn
where the $\tilde Y_i$ are mutually independent and each
one is distributed identically to its corresponding
$Y_i$.
This  completes the proof.

\begin{remark}
   
The lower bound is
only asymptotic (rather than {\em finite-sample}, in the sense of holding for all $n$)
--- necessarily so. 
This is because even for a single binomial $Y\sim \Bin(n,p)$, the behavior of 
$\E|Y-np|$ is roughly $np(1-p)$
for $p\notin[1/n,1-1/n]$ and $\approx \sqrt{np(1-p)}$ elsewhere 
\cite[Theorem 1]{Berend2013}.
This precludes any
finite-sample
lower bound of the form 
$\E\nrm{p-\hat p}_\infty \ge c\sqrt{V(p)/n}$.
\end{remark}

\subsection{Proof of Proposition \ref{prop:p_0-lb}}\label{proof7}
It follows 
from Lemma~\ref{lem:bern-np}
that $\E|p-\hat p|\ge c\sqrt{p_0/n}$
for some universal $c>0$. Since
$
\E\sup_{i\in\N}|p_i-\hat p_i|
\ge
\sup_{i\in\N}\E|p_i-\hat p_i|
$, 
this  completes the proof.

\subsection{ A Proof for Theorem \ref{selective_inference_LB}}\label{proof8}

We begin with the following proposition.
\begin{proposition}\label{te}
     Assume there exists $V_\delta(X^n)$ such that
    \begin{align}\label{most_frequent_2}
        \P\left(|p_j-\hat{p}_j|\geq V_\delta(X^n) |p_j=p_{[1]} \right) \leq \delta.
    \end{align}
    Then, 
$$\mathbb{E}(V_\delta(X^n))\geq z_{\delta/2}\sqrt{\frac{p_{[1]}(1-p_{[1]})}{n}}+O\left(\frac{1}{n}\right).$$
\end{proposition}
\begin{proof}
Assume there exists $V_\delta(X^n)$ that satisfies (\ref{most_frequent_2}) and 
    $$\mathbb{E}(V_\delta(X^n))< z_{\delta/2}\sqrt{\frac{p_{[1]}(1-p_{[1]})}{n}}+O\left(\frac{1}{n}\right).$$
From (\ref{most_frequent_2}), we have that 
\begin{align}\label{cond}
     \P\left(|p_j-\hat{p}_j|\geq U_\delta(X^n) |p_j=p_{[1]} \right)=\P\left(|p_{[1]}-\hat{p}_j|\geq U_\delta(X^n) |p_j=p_{[1]} \right)\leq \delta.
\end{align}
Now, consider $Y\sim \text{Bin}(n,p_{[1]})$. Let $Y^n$ be a sample of $n$ independent observations. Notice we can always extend the Binomial case to a multinomial setup with parameters $p$, over any alphabet size $||p||_0$. That is, given a sample $Y^n$, we may replace every $Y=0 $ (or $Y=1$) with a sample from a multinomial distribution over an alphabet size $||p||_0-1$. Further, we may focus on samples for which $p_{[1]}$ is the most likely event in the alphabet, and construct a CI for $p_{[1]}$ following (\ref{cond}). This means that we found a CI for $p_{[1]}$ with an expected length that is shorter than the CP CI, which contradicts its optimality.  
 
\end{proof}
Now, assume there exists $U_\delta(X^n)$ that satisfies 
\begin{align}
        \P\left(|p_j-\hat{p}_j|\geq U_\delta(X^n)  \right) \leq \delta.
    \end{align}
and 
\begin{align}\label{temp}
\mathbb{E}(U_\delta(X^n))< z_{\delta/2}\sqrt{\frac{p_{[1]}(1-p_{[1]})}{n}}+O\left(\frac{1}{n}\right).
\end{align}
For simplicity of notation, denote $v=\argmax_i p_i$ as the symbol with the greatest probability in the alphabet. That is, $p_v=p_{[1]}$. We implicitly assume that $v$ is unique, although the proof holds in case of several maxima as well. We have that
\begin{align}\label{te2}
    \P\left(|p_j-\hat{p}_j|\geq U_\delta(X^n)  \right)=&\sum_{u\in \mathcal{X}}\P\left(|p_j-\hat{p}_j|\geq U_\delta(X^n)  |j=u\right)\P(j=u)=\\\nonumber
    & \P\left(|p_{[1]}-\hat{p}_j|\geq U_\delta(X^n)  |j=v\right)\P(j=v)+\\\nonumber
    &\sum_{u \neq v}\P\left(|p_j-\hat{p}_j|\geq U_\delta(X^n)  |j=u\right)\P(j=u).
\end{align}
Proposition \ref{te} together with  assumption (\ref{temp}) imply that $$\P\left(|p_{[1]}-\hat{p}_j|\geq U_\delta(X^n)  |j=v\right)>\delta.$$ On the other hand, it is well-known that $\hat{p}_{[1]} \rightarrow p_{[1]}$ for sufficiently large $n$ \citep{gelfand1992inference,shifeng2005testing,xiong2009inference}. This means that $\P(j=u)\rightarrow 1$ and (\ref{te2}) is bounded from below by $\delta$, for sufficiently large $n$. This contradicts (\ref{cond}) as desired.

\section*{Appendix A}

We show that 
$$\sup_{p \in [0,1-1/n]} \big|\left(p(1-p)\right)^k-\left((p+1/n)(1-(p+1/n))\right)^k)\big|\leq \frac{k}{n\cdot 4^{k-1}}+\frac{3k(k-1)(k-2)}{n^3\cdot 2^{2k-5}}
$$
Let $0\leq p \leq 1/2-1/n$. Denote $f_k(p)=\left((p(1-p)\right)^k$. Applying Taylor series to $f_k(p+1/n)$ around $f_k(p)$ yields
$$f_k\left(p+\frac{1}{n}\right)=f_k(p)+\frac{1}{n}f'_k(p)+r(p)$$
where $r(p)=\frac{1}{3!}\frac{1}{n^3}f'''(c)$ is the residual and $c \in [p, p+1/n]$ \citep{stromberg2015introduction}. 
We have
\begin{align}\label{derivatives}
    &f'_k(p)=k\left(p(1-p)\right)^{k-1}(1-2p)\leq k \left(p(1-p)\right)^{k-1}\\\nonumber
    &f'''_k(p)=k(k-1)(k-2)p^{k-3}(1-p)^{k-3}(1-2p)^3-6k(k-1)p^{k-2}(1-p)^{k-2}(1-2p)\leq\\\nonumber
    &\quad \quad \quad \;\;\;  k(k-1)p^{k-3}(1-p)^{k-3}\left((k-2)+6p(1-p)\right).
\end{align}
Hence,
\begin{align}
&\sup_{p \in [0,1/2-1/n]} \big|\left(p(1-p)\right)^k-\left((p+1/n)(1-(p+1/n))\right)^k)\big|=\\\nonumber
&\sup_{p \in [0,1/2-1/n]}\big|-\frac{1}{n}f_k'(p)-\frac{1}{3!}\frac{1}{n^3}f'''(c)\big|\leq\sup_{p \in [0,1/2-1/n]}\frac{1}{n}\big|f_k'(p)\big|+\frac{1}{3!}\frac{1}{n^3}\big|f'''(c)\big|\overset{(\text{i})}{\leq}\\\nonumber
&\sup_{p \in [0,1/2-1/n]}\frac{k}{n}\left(p(1-p)\right)^{k-1}+k(k-1)p^{k-3}(1-p)^{k-3}\left((k-2)+6p(1-p)\right)\overset{(\text{ii})}{\leq}\\\nonumber
&\quad\quad\quad\quad\;\;\frac{k}{n\cdot 4^{k-1}}+\frac{3k(k-1)(k-2)}{n^3\cdot 2^{2k-5}}
\end{align}
where 
\begin{enumerate}
    \item follows from (\ref{derivatives}).
    \item follows from the concavity of $\left(p(1-p)\right)^{k}$ for $k\geq 1$. 
\end{enumerate}

\section*{Appendix B}
We study $\min_m m/a^{1/m}$ for some positive $a$. This problem is equivalent to 
$$\min_m \log(m)-\frac{1}{m}\log(a).$$
Taking its derivative with respect to $m$ and setting it to zero yields
$$\frac{d}{dm}\log(m)-\frac{1}{m}\log(a)=\frac{1}{m}+\frac{1}{m^2}\log(a)=0.$$
Hence, $m^*=\log(1/a)$. Therefore, 
\begin{align}
\min_m m/a^{1/m}=&\exp(\log(m^*)-({1}/{m^*})\log(a))=\exp(1)\log(1/a).
\end{align}

\section*{Appendix C}
We study 
    \begin{align}
        \min_{m \in \mathbb{R}^+} \left(\frac{\sqrt{m/2}}{\delta^{1/m}}\right)\exp\left(-\frac{1}{2}+\frac{1}{m}\right) 
    \end{align}
This problem is equivalent to 
    \begin{align}
        \min_{d \in \mathbb{R}^+} 
        \frac{1}{2}\log(d)+\frac{1}{2d}\log\left(\frac{1}{\delta}\right)-\frac{1}{2}+\frac{1}{2d}
    \end{align}
where $d=m/2$. Taking its derivative with respect to $d$ and setting it to zero yields
$$\frac{1}{2d}-\frac{1}{2d^2}\left(\log\left(\frac{1}{\delta}\right)+1\right)  =0.$$
Hence, $d^*=\log(1/
\delta)+1$. Therefore, 
    \begin{align}
        \min_{d \in \mathbb{R}^+} 
        \frac{1}{2}\log(d)+\frac{1}{2d}\log\left(\frac{1}{\delta}\right)-\frac{1}{2}+\frac{1}{2d}=\frac{1}{2}\log(\log(1/
\delta)+1)
    \end{align}
and
 \begin{align}
        \min_{m \in \mathbb{R}^+} \left(\frac{\sqrt{m/2}}{\delta^{1/m}}\right)\exp\left(-\frac{1}{2}+\frac{1}{m}\right)=\sqrt{\log\left(\frac{1}{\delta}\right)+1}. 
    \end{align}

\section*{Appendix D}
\begin{proposition}
    Let $p_{i \in \mathbb{N}}$ be a probability distribution over $\mathbb{N}$. Then, 
    \begin{align}\label{A22}
        p_{[1]}=\max_{i\in \mathbb{N} }p_i(1-p_i)
    \end{align} where $p_{[1]}=\max_{i \in \mathbb{N}}p_i$ is the largest element in $p$. 
\end{proposition}
\begin{proof}
Let us first consider the case where $p_i\leq 1/2$ for all $i \in \mathbb{N}$. Then (\ref{A22}) follows directly from the montonicity of $p_i(1-p_i)$ for $p_i \in [0,1/2]$.  
Next, assume there exists a single $p_j>1/2$. Specifically, $p_j=1/2+a$ for some positive $a$. Then, the remaining $p_i$'s are necessarily smaller than $1/2$. Further, the maximum of $p_i(1-p_i)$ over $i\neq j$ is obtained for $p_i=1/2-a$, from the same monotonicity reason. This means that $\max_{i \neq j }p_i(1-p_i)=(1/2-a)(1-(1/2-a))=(1/2+a)(1-(1/2+a))$ where the second equality follows from the symmetry of $p_i(1-p_i)$ around $p_i=1/2$, which concludes the proof. 
\end{proof}

\bibliographystyle{plainnat}
\bibliography{bibi}  
\end{document}